\documentclass[11pt, twoside]{article}
\usepackage[utf8]{inputenc}
\usepackage[margin=2.5cm]{geometry} % Adjust the margin size
\usepackage{graphicx}
\usepackage{soul}
\usepackage{lipsum}
\usepackage{fancyhdr} % Added for header customization
\usepackage{array}
\usepackage{tikz}
\usepackage{amsmath,amssymb,amsfonts}
\usepackage{algorithm}
\usepackage[export]{adjustbox}
\usepackage[absolute,overlay]{textpos}

\usepackage{verbatim}
\usepackage{multicol}
\usepackage{mathtools}
\usepackage{mathrsfs}
\usepackage{subcaption}
\usepackage{bm}
\usepackage{amsthm}
\usepackage{amsmath,amssymb,amsfonts}
\usepackage{gensymb}
\usepackage{enumerate}
\usepackage{xcolor,colortbl}
\usepackage{algorithmic}
\usepackage[pdfpagemode={UseOutlines},bookmarks=true,bookmarksopen=true,
   bookmarksopenlevel=0,bookmarksnumbered=true,hypertexnames=false,
   colorlinks,linkcolor={blue},citecolor={blue},urlcolor={red},
   pdfstartview={FitV},unicode,breaklinks=true]{hyperref}
\definecolor{lightgray}{rgb}{0.85,0.85,0.85}

\newcommand{\sC}			{{\mathscr{C}}}
\newcommand{\sU}			{{\mathscr{U}}}
\newcommand{\sD}            {{\mathscr{D}}}

\newcommand{\sW}            {{\mathscr{W}}}

\newcommand{\cH}              {{\mathcal{H}}}
\newcommand{\cR}            {{\mathcal{R}}}

\newcommand{\cA}            {{\mathcal{A}}}

\newcommand{\cP}            {{\mathcal{P}}}
\newcommand{\cX}         {{\mathcal{X}}}
\newcommand{\cY}            {{\mathcal{Y}}}
\newcommand{\cK}         {{\mathcal{K}}}

\newcommand{\cZ}         {{\mathcal{Z}}}

\newcommand{\nn}			{{\nonumber}}
\newcommand{\R}			{{\mathbb{R}}}
\newcommand{\A}         {{\mathbb{A}}}
\newcommand{\U}         {{\mathbb{U}}}
\newcommand{\D}         {{\mathbb{D}}}
\newcommand{\B}         {{\mathbb{B}}}
\newcommand{\N}         {{\mathbb{N}}}

\newcommand{\I}         {{\mathbb{I}}}

\newcommand{\W}         {{\mathbb{W}}}

\newcommand{\norm}[1]{\left\lVert#1\right\rVert}
\DeclareMathOperator*{\argmin}{arg\,min}
\DeclareMathOperator*{\argmax}{arg\,max}

\newcommand{\iprod}[2]{\left\langle#1, #2\right\rangle}
\newcommand{\half}[0]{\frac{1}{2}}
\newcommand{\integral}[4]{\int_{#1}^{#2}{#3}\,d#4}

\newtheorem{definition}{Definition}
\newtheorem{theorem}{Theorem}
\newtheorem{corollary}{Corollary}
\newtheorem{remark}{Remark}
\newtheorem{assumption}{Assumption}
\newtheorem{lemma}{Lemma}

\title{\textbf{Upper and Lower Bounds for a Class of Constrained Linear Time-Varying Games}\thanks{This research was supported by the Australian Commonwealth Government and the Faculty of Engineering and Information Technology (University of Melbourne) through the Ingenium Scholarship and the Air Force Office of Scientific Research (Grant number: FA2386‐22-1-4074).}}
\author{Vincent Liu$^\dag$ \and Chris Manzie$^\dag$ \and Peter M. Dower\thanks{The authors are with the Department of Electrical and Electronic Engineering, University of Melbourne, VIC, 3010 Australia (email: liuv2@student.unimelb.edu.au; manziec@unimelb.edu.au; pdower@unimelb.edu.au).}}
\date{}

\begin{document}
\maketitle
\begin{abstract}
This paper develops an algorithm for upper- and lower-bounding the value function for a class of linear time-varying games subject to convex control sets. In particular, a two-player zero-sum differential game is considered where the respective players aim to minimise and maximise a convex terminal state cost. A collection of solutions of a single-player dynamical system subject to a trimmed control set is used to characterise a viscosity supersolution of a Hamilton-Jacobi (HJ) equation, which in turn yields an upper bound for the value function. Analogously, a collection of hyperplanes is used to characterise a viscosity subsolution of the HJ equation, which yields a lower bound. The computational complexity and memory requirement of the proposed algorithm scales with the number of solutions and hyperplanes that characterise the bounds, which is not explicitly tied to the number of system states. Thus, the algorithm is tractable for systems of moderately high dimension whilst preserving rigorous guarantees for optimal control and differential game applications.

\vspace{1em}
\end{abstract}

\section{Introduction}
\label{sec: introduction}
The concept of a value function arises in various control and estimation problems including the synthesis of optimal feedback control laws and filters \cite{kunisch04, fleming2000max}, the modelling of differential games \cite{ES:84}, the design of optimal trajectories \cite{skoraczynski2023incentivizing}, and the characterisation of reachable sets \cite{chen2018hamilton}. Though versatile in its applications, with the exception of specific problems such as linear quadratic regulators, easy-to-compute characterisations of the value function are not readily available. Under mild assumptions, the value function admits an equivalent characterisation as the solution of a Hamilton-Jacobi (HJ) equation and conventional grid-based approaches such as level set methods \cite{mitchell2005time}, Markov chain approximations \cite{kushner1990numerical}, and semi-Lagrangian schemes \cite{falcone2006numerical} provide a general purpose means of numerically solving these HJ equations. However, these approaches suffer from the curse-of-dimensionality, which typically limits their use to system with no more than four states \cite{bokanowski2022optimistic}.

Modern literature on numerical methods for solving HJ equations largely focus on mitigating this well-known curse-of-dimensionality. In the special case where the Hamiltonian is state-independent, a generalised Hopf formula may be used to characterise solutions of HJ equations via a finite-dimensional nonlinear program \cite{chow2018algorithm}, which yields a highly parallelisable means of computing the value function. Otherwise, for systems where the control input is of lower dimension than the state, optimistic planning algorithms \cite{bokanowski2022optimistic} and tree-structure approaches \cite{alla2019efficient} can yield efficient approximations of the value function as discretisation is performed in the control space as opposed to the state space. In a more general setting, parametrisation of the value function via a finite collection of basis functions, such as in max plus methods \cite{akian2008max}, neural networks \cite{DLM20}, and polynomial approximations \cite{kalise2018polynomial}, may also yield computationally tractable approximation schemes for systems of moderate dimension. 

Provable properties for approximations of the value function are often desirable for a range of applications, yet such properties can be difficult to obtain for the aforementioned numerical schemes. For instance, in reachability analysis, upper and lower bounds for the value function may be used to produce inner- and outer-approximations of reachable sets \cite{liu2024hamilton}. To this end, classic comparison results for viscosity solutions of HJ equations (see, e.g., \cite[Thm 5.2.7, p.109]{CS:04}) can facilitate the construction of verifiable bounds for the value function. In the case of time-invariant polynomial systems, polynomial supersolutions and subsolution of HJ equations can be generated via a sum-of-squares program, which yields upper and lower bounds for the value function \cite{xue2019inner, jones2019}. Similar ideas are utilised in \cite{sinyakov2015method} wherein super- and subsolutions are characterised for bilinear systems in which the state matrix lies within a matrix ellipsoid.

In this paper, an algorithm is developed for upper- and lower-bounding the value function for a class of linear time-varying (LTV) games subject to convex control sets. In particular, a two-player zero-sum differential game is considered where the respective players aim to minimise and maximise a terminal state cost, which is assumed to be convex. The proposed upper and lower bounds are, respectively, a viscosity supersolution and subsolution of an HJ equation associated with the differential game. The notion of viscosity solution is explicitly utilised in this work to construct bounds that are not continuously differentiable. The proposed upper bound in this work is a generalisation of \cite[Section III.b]{liu2024efficient} to a differential game setting. Compared to \cite{liu2024efficient}, one additional assumption is imposed, which roughly requires that the minimising player has greater control authority than the maximising player. 

The proposed upper bound is characterised by a collection of solutions of an auxiliary single-player system whose controls belong to a `trimmed' control set. If these solutions are optimal with respect to minimising the terminal state cost, the upper bound coincides with the value function along these solutions, yielding a tight upper bound. When required, the values of the upper bound can be accessed by solving a quadratic program. The proposed lower bound is characterised by the maximum of a collection of hyperplanes, which can be made to coincide with the value function along the same set of solutions that characterise the upper bound. Thus, the upper and lower bounds provide an interval in which the value function must lie and this interval collapses along a characteristic set of solutions. Both the computational complexity and memory requirement of the proposed bounds are tied to the sparsity of the collection of solutions and hyperplanes, which is independent of the number of system states. Note that hyperplanes have been previously used to construct outer-approximations of reachable sets for LTV games in \cite{hwang2005polytopic}, although a characterisation as a lower bound for the value function is not provided. This work extends on \cite{hwang2005polytopic} by generalising the control sets, characterising the maximum of the hyperplanes as a viscosity subsolution, and demonstrating tightness of the lower bound. In \cite{akian2025stochastic}, a collection of hyperplanes have also been used to lower bound the value function, albeit for discrete-time nonlinear systems in a single-player setting. 

The remainder of this paper is organised as follows. Section \ref{sec: preliminaries} presents the problem setting and reviews the necessary preliminaries for the two-player differential game considered in this work. Section \ref{sec: main results} demonstrates the upper and lower bound guarantees of the proposed approximating functions. The application of these results to reachability analysis is discussed in Section \ref{sec: reachability}. Section \ref{sec: algorithm} then describes an algorithmic procedure for computing the proposed upper and lower bounds. Finally, an illustrative example is presented in Section \ref{sec: numerical example}.

\section*{Notation}

 Let $\N\doteq \{1,2,\ldots\}, \R$, and $\R_{\geq 0}$ denote the set of natural, real, and non-negative real numbers, respectively.  Let $\langle \cdot, \cdot\rangle$ denote the Euclidean inner product with $\norm{\cdot}\doteq \iprod{\cdot}{\cdot}^\half$ denoting the corresponding norm. The distance $d(\cdot;\,\Omega):\R^n\rightarrow\R$ to a given set $\Omega \subset \R^n$ is defined by $d(x;\,\Omega) \doteq \inf_{\xi\in\Omega}\norm{x - \xi}$ for all $x\in\R^n$. The space of continuous functions from $\Omega$ to $\cY$ is denoted by $\sC(\Omega;\,\cY)$. The pre-image of a set $\cY_0\subseteq \cY$ under $g:\Omega \rightarrow \cY$ is denoted by $g^{-1}\left(\cY_0 \right) \doteq \{ \omega \in \Omega\, |\, g(\omega) \in \cY_0 \}$. An open ball of radius $r > 0$ centred at $x_0 \in \R^n$ is denoted by $\B_r(x_0)\doteq \{x\in\R^n\,|\,\norm{x - x_0} < r\}$. The convex hull of a set of $N$ points $\{x_i\}^{N}_{i=1}$ is denoted by $\text{conv}(\{x_i\}^{N}_{i=1}) \doteq \big\{\sum^{N}_{i=1}\lambda_ix_i\,|\,\sum^{N}_{i=1}\lambda_i = 1,\, \{\lambda_i\}^{N}_{i=1}\subset[0,1] \big\}$. The complement of a set $\cY\subseteq \R^n$ is denoted by $\cY^c\doteq \{y\in \R^n\,|\,y\notin\cY\}$. The Minkowski sum and difference of the sets $\cY\subset \R^n$ and $\cZ\subset \R^n$ are denoted by $\cY\oplus\cZ\doteq \{y + z\in\R^n\,|\,y \in \cY, z\in \cZ\}$ and $\cY \ominus \cZ \doteq \{c\in\R^n\,|\,\{c\}\oplus \cZ \subseteq \cY\}$, respectively. The image of a set $\cY \subset \R^m$ under a linear map $ y\mapsto Py$ with $P\in \R^{n\times m}$ is denoted by $P\cY \doteq \{Py \in \R^n\,|\,y\in\cY\}$. The transpose of a matrix $P$ is denoted by $P'$ and the $n$-by-$n$ identity matrix is denoted by $\I_n$.

\section{Preliminaries} 
\label{sec: preliminaries}
Fix $T \geq t \geq 0 $, $x\in\R^n$, and consider a continuous-time LTV system with trajectories satisfying the finite-dimensional Cauchy problem
\begin{equation}
    \begin{split}
    \dot{\xi}(s) &= A(s)\xi(s) + B(s)u(s) + E(s)d(s), \quad \text{a.e. } s \in (t,T), \\
    \xi(t) &= x,
    \end{split}
    \label{eq: prelim: linear system}
\end{equation}
in which $\xi(s)\in\mathbb{R}^n$ is the state, $u(s)\in\U\subset\R^{m}$ is the input for player I, and $d(s)\in\D\subset\R^{\ell}$ is the input for player II, all at time $s$ with $A\in\sC([0,T];\,\R^{n\times n})$, $B\in\sC([0,T];\,\R^{n\times m})$, and $E\in\sC([0,T];\,\R^{n\times \ell})$. The control inputs for player I and II are selected such that $u\in\sU[t,T]$ and $d\in\sD[t,T]$ with
\begin{equation}
    \begin{split}
        \sU[t,T] &\doteq \left\{u:[t,T]\rightarrow \U\;| \;u \text{ measurable} \right\}, \\
        \sD[t,T] &\doteq \left\{d:[t,T]\rightarrow \D\;| \;d \text{ measurable} \right\},
    \end{split}
    \label{eq: prelim: control sets for player I and II}
\end{equation}
and both $\U$ and $\D$ are compact, convex, and non-empty. For any fixed controls $u\in\sU[t,T]$ and $d\in\sD[t,T]$, \eqref{eq: prelim: linear system} admits a unique and continuous solution, which will be denoted by $\xi^{u,d}_{t,x}(s)$ at time $s\in[t,T]$. 

To introduce the differential game setting used in this work, definitions for the strategies that each player may employ are recalled in the following from \cite[Def 1.1, p.433]{bardi1997optimal} and \cite{ES:84}.
\begin{definition} A strategy for player I (resp. player II) is a map $\alpha:\sD[t,T]\rightarrow\sU[t,T]$ (resp. $\delta:\sU[t,T]\rightarrow\sD[t,T]$). The strategy is non-anticipative if for any $s\in[t,T]$ and $d_1, d_2 \in \sD[t,T]$ (resp. $u_1, u_2\in\sU[t,T]$)
\begin{equation}
    \begin{split}&d_1(\tau) = d_2(\tau) \text{ a.e. } \tau \in [t, s] \implies \alpha[d_1](\tau) = \alpha[d_2](\tau) \text{ a.e. } \tau \in [t, s] \\
         &u_1(\tau) = u_2(\tau) \text{ a.e. } \tau \in [t, s] \implies \delta[u_1](\tau) = \delta[u_2](\tau) \text{ a.e. } \tau \in [t, s]. \quad \text{(resp.)}
    \end{split}
\end{equation}
\label{def: prelim: non-anticipative strategy}
\end{definition}

In this work, player I is assumed to have no access to information about player II's actions, thus, they employ a strategy that is independent of player II's control. Conversely, player II is assumed to have access to player I's actions up until the current time. As such, player II may utilise a non-anticipative strategy (see Definition \ref{def: prelim: non-anticipative strategy}). These assumptions are summarised in the following. 
\begin{assumption} 
\begin{enumerate}[(i)]
    \item Player I employs a strategy $\alpha:\sD[t,T]\rightarrow\sU[t,T]$ that is independent of player II's control, i.e., $\alpha[d] = u\in\sU[t,T]$ for all $d\in\sD[t,T]$.
    \item Player II employs a non-anticipative strategy and the set of all non-anticipative strategies for player II is denoted by $\Delta[t,T]$. 
\end{enumerate}
    \label{assumption: prelim: strategies}
\end{assumption}

Attach to \eqref{eq: prelim: linear system} the value function $v: [0,T]\times \mathbb{R}^n\rightarrow \R$ defined by
\begin{equation}
    v(t,x) \doteq \sup_{\delta \in \Delta[t,T]}\inf_{u\in\sU[t,T]}g\left(\xi^{u,\delta[u]}_{t,x}(T)\right),
    \label{eq: prelim: value function}
\end{equation}
in which $g\in\sC(\R^n;\,\R)$ is a Mayer cost (see \cite[Chap 7.1, p.186]{CS:04}) for the zero-sum differential game between player I and II who wish to minimise and maximise $g$, respectively, at the terminal time $T$. 

\begin{remark}
The value in \eqref{eq: prelim: value function} is also referred to as the upper value of the differential game due to the information advantage given to player II. In the problem setting of this work, Isaacs' condition is satisfied, which implies that if player I were given the advantage instead, i.e., the role of each player is reversed in Assumption \ref{assumption: prelim: strategies}, the value of the game would not change (see \cite[Cor 2.2, p.443]{bardi1997optimal}). Thus, the choice of favouring player II is arbitrary and the results of this work still hold when favouring player I.
\label{remark: prelim: equivalence of lower and upper values}
\end{remark}

\subsection{Viscosity Solutions}
\label{sec: visc solutions}
The notion of viscosity solution is needed to construct upper and lower bounds for $v$ that are not necessarily continuously differentiable. Accordingly, the relevant background for viscosity solutions is reviewed and adapted from \cite{bardi1997optimal} and \cite{crandall1983viscosity}. Consider the HJ equation given by
\begin{equation}
    -v_t(t,x) + H(t,x,\nabla v(t,x)) = 0, \quad \forall (t,x) \in (0,T)\times\R^{n},
    \label{eq: prelim: HJI}
\end{equation}\normalsize
in which the Hamiltonian $H\in\sC([0,T]\times\R^ n\times \R^n;\,\R)$ in this work is given by
\begin{equation}
    H(t,x, p) \doteq \max_{u\in\U}\min_{d\in\D}\iprod{-p}{A(t)x+B(t)u + E(t)d}.
     \label{eq: prelim: hamiltonian}
\end{equation}

\begin{definition}
A pair $(q, p)\in\R\times \R^{n}$ is a subgradient (resp. supergradient) of $v\in\sC\left([0,T]\times\R^n;\,\mathbb{R}\right)$ at $(t_0, x_0)\in(0,T)\times\R^{n}$ if
\begin{align}
    \liminf_{(t,x)\rightarrow(t_0,x_0)}\frac{v(t,x)-v(t_0,x_0)-\iprod{(q, p)}{(t,x)-(t_0,x_0)}}{\norm{(t,x)-(t_0,x_0)}}  \geq 0,& \label{eq: prelim: subgradient def}\\
    \limsup_{(t,x)\rightarrow(t_0,x_0)}\frac{v(t,x)-v(t_0,x_0)-\iprod{(q,p)}{(t,x)-(t_0,x_0)}}{\norm{(t,x)-(t_0,x_0)}}\leq 0.& \quad (\text{resp.}) \label{eq: prelim: supergradient def}
\end{align}
The set of all $(q,p)\in\R\times \R^{n}$ satisfying \eqref{eq: prelim: subgradient def} (resp. \eqref{eq: prelim: supergradient def}) is called the subdifferential (resp. superdifferential) of $v$ at $(t_0, x_0)$, which is denoted by $D^-v(t_0, x_0)$ (resp. $D^+v(t_0,x_0)$).
\label{def: prelim: reg subdifferentials and superdifferentials}
\end{definition}

Viscosity solutions of \eqref{eq: prelim: HJI} may be defined via these subdifferentials and superdifferentials (see \cite[Def 1.1, p.25]{bardi1997optimal} and \cite[Lem 1.7, p.29]{bardi1997optimal}).

\begin{definition}
A function $v \in \sC\left([0,T]\times\R^{n};\,\R\right)$ is a continuous viscosity supersolution (resp. subsolution) of \eqref{eq: prelim: HJI} if for all $(t_0,x_0)\in(0,T)\times\R^{n}$ and for all $(q,p)\in D^-v(t_0, x_0)$ (resp. $D^+v(t_0,x_0)$),
\begin{equation}
    -q + H(t_0,x_0, p) \geq 0 \quad (\text{resp. } \leq 0).
    \label{eq: prelim: visc supersolution and subsolution condition}
\end{equation}
Additionally, $v$ is a continuous viscosity solution of \eqref{eq: prelim: HJI} if it is both a viscosity supersolution and subsolution of \eqref{eq: prelim: HJI}.
\label{def: prelim: visc solution}
\end{definition}  

The value $v$ in \eqref{eq: prelim: value function} admits an equivalent characterisation as the unique viscosity solution of \eqref{eq: prelim: HJI} subject to an appropriate terminal condition, which follows from \cite[Thm 4.1, Cor 4.2]{ES:84}. Note that the boundedness and Lipschitz properties of $g$ in (2.2) of \cite{ES:84} may be relaxed if these same properties are not required of $v$ \cite[Thm 3.2]{ES:84}. 
\begin{theorem} Consider the terminal value problem
\begin{equation}
\begin{split}
    -v_t(t,x) + H(t, x, \nabla v(t,x)) = 0, \quad &\forall (t,x)\in (0,T)\times\mathbb{R}^n, \\
    v(T,x) = g(x), \quad &\forall x\in\mathbb{R}^n,
    \end{split}
    \label{eq: prelim: HJI with terminal condition}
\end{equation}
in which $g\in\sC\left(\R^n;\,\R\right)$ and the Hamiltonian $H$ is given by \eqref{eq: prelim: hamiltonian}. Then, \eqref{eq: prelim: HJI with terminal condition} admits the unique viscosity solution $v\in\sC([0,T] \times\R^n;\,\R)$ given by \eqref{eq: prelim: value function}.
\label{theorem: prelim: value function and HJI equation}
\end{theorem}

The linchpin in the construction of the desired upper and lower bounds of $v$ in this work, is a classic comparison result for solutions of \eqref{eq: prelim: HJI}, which follows from \cite[Ex 3.6, p.182]{bardi1997optimal} and \cite[Rem 3.10, pp.154-155]{bardi1997optimal}.

\begin{theorem} Let $\Bar{v},\underline{v} \in \sC([0,T]\times\R^{n};\,\R)$ be, respectively, a viscosity supersolution and subsolution of \eqref{eq: prelim: HJI} with the Hamiltonian $H$ given by \eqref{eq: prelim: hamiltonian}. If $\Bar{v}(T,x) \geq \underline{v}(T,x)$ for all $x\in \R^n$, then, $\Bar{v}(t,x) \geq \underline{v}(t,x)$ for all $(t,x)\in [0,T]\times\R^n$.
\label{theorem: prelim: comparison result}
\end{theorem}

The following assumptions will be considered throughout this work with the relevant set of assumptions explicitly listed in each result. 
\begin{assumption} $g\in\sC(\R^n;\,\R)$ is convex and Lipschitz continuous with Lipschitz constant $L_g \geq 0$.
    \label{assumption: prelim: Mayer cost condition}
\end{assumption}
\vspace{-1.5em}
\begin{assumption} The set 
\begin{equation}
    \W(t) \doteq B(t)\U \ominus(- E(t)\D)
    \label{eq: prelim: trimmed input set def}
\end{equation}
is non-empty for all $t\in[0,T]$. 
    \label{assumption: prelim: trimmed input condition}
\end{assumption}
\begin{assumption} The set $\W(t)$ as defined in \eqref{eq: prelim: trimmed input set def} satisfies
\begin{equation}
    \W(t)\oplus \left(-E(t)\D\right) = B(t)\U
    \label{eq: prelim: trimmed input set condition}
\end{equation}
for all $t\in[0,T]$. 
    \label{assumption: prelim: recoverability of player 1 set}
\end{assumption}

\begin{remark}
    In general, $\W(t)\oplus\left(-E(t)\D\right)\subseteq B(t)\U$, which follows by definition of the Minkowski difference. However, both Assumptions \ref{assumption: prelim: trimmed input condition} and \ref{assumption: prelim: recoverability of player 1 set} hold in the case where $B(t)\U$ and $-E(t)\D$ are zonotopes with the latter set being \emph{aligned} with the former for all $t\in[0,T]$ (see \cite[Def 1, Prop 3]{yang:21}). In particular, suppose 
    \begin{equation*}
    B(t)\U = \Big\{c_1(t) + \sum^{n_p}_{i=1}\theta_i b_i(t)\,\big |\,\theta_i\in[-1,1]\text{ for all } i\in\{1,\ldots,n_p\} \Big\},
    \end{equation*}
    in which $c_1(t)\in\R^n$ and $b_i(t)\in\R^n$ for all $i\in\{1,\ldots,n_p\}$. Then, $-E(t)\D$ is aligned with $B(t)\U$ if
    \begin{equation*}
            -E(t)\D  = \Big\{c_2(t) + \sum^{n_p}_{i=1}\theta_i \kappa_i(t)b_i(t)\,\big|\,\theta_i\in[-1,1]\text{ for all } i\in\{1,\ldots,n_p\} \Big\},
    \end{equation*}
    in which $c_2(t)\in\R^n$ and $\kappa_i(t)\in[0,1]$ for all $i\in\{1,\ldots,n_p\}$. 
    \label{remark: prelim: trimmed input set conditions}
\end{remark}
\section{Main Results}
\label{sec: main results}
 
In this section, semi-explicit characterisations for upper and lower bounds for the value function $v$ in \eqref{eq: prelim: value function} are constructed. In particular, a candidate upper bound $\Bar{v}\in\sC([0,T]\times\R^n;\,\R)$ is proposed with a terminal condition $\Bar{v}(T,x) \doteq  \Bar{g}(x) $ that upper bounds the cost $g$. The upper bound guarantees are then obtained by demonstrating that $\Bar{v}$ is a viscosity supersolution of \eqref{eq: prelim: HJI}, which follows by Theorem \ref{theorem: prelim: comparison result} and noting that the value function $v$ is a viscosity subsolution of \eqref{eq: prelim: HJI}. A lower bound $\underline{v}\in\sC([0,T]\times\R^n;\,\R)$ is constructed in an analogous manner by demonstrating that $\underline{v}$ is a viscosity subsolution of \eqref{eq: prelim: HJI} with a terminal condition $\underline{v}(T,x) \doteq \underline{g}(x) $ that lower bounds $g$.

From the perspective of (the minimising) player I, an upper bound for $v$ yields a conservative estimate for the minimum value of $g$ that can be attained at time $T$ despite any strategy $\delta\in\Delta[t,T]$ that player II may employ. Analogously for (the maximising) player II, a lower bound for $v$ yields a conservative estimate for the maximum value of $g$ that can be attained despite the presence of an adversarial input from player I. 

\subsection{Upper-bounding the Value Function}
\label{sec: upper-bounding the value function}

To produce an upper bound for the Mayer cost $g$, consider a collection of $N\in\N$ levels $\{\gamma_k\}^N_{k=1}\subseteq g\left(\R^n\right)$, which implies that the $\gamma_k$-level set of $g$, i.e., $g^{-1}(\{\gamma_k\})$, is non-empty. For each level set, attach a set of $n_k\in\N$ (potentially non-unique) points $\{\Bar{x}_{i,k}\}^{n_k}_{i=1}\subseteq g^{-1}\left(\{\gamma_k\}\right)$ for all $i\in\{1,\ldots,n_k\}$. Consider the set
\begin{equation}
    \Bar{\Omega}_k \doteq \text{conv}\left(\{\Bar{x}_{i,k}\}^{n_k}_{i=1}\right).
    \label{eq: main: convex hull of terminal states}
\end{equation}
Then, the collection of points $\{\Bar{x}_{i,k}\}^{n_k, N}_{i,k=1}$ yields a candidate upper bound $\Bar{g}:\R^n\rightarrow \R$ defined by
\begin{align}
    \Bar{g}(x) &\doteq \min_{k\in\{1,\ldots,N\}}\Bar{g}_k(x), \label{eq: main: upper bound for g}\\
    \Bar{g}_k(x) &\doteq L_g d(x;\,\Bar{\Omega}_k) + \gamma_k,
    \label{eq: main: upper bound for g via each level set}
\end{align}
in which $L_g\geq 0$ is the Lipschitz constant of $g$ and $d(\cdot;\,\Bar{\Omega}_k)$ is the distance to $\Bar{\Omega}_k$. An illustrative depiction of $\Bar{g}$ for a 1D example is presented in Fig. \ref{fig: main: upper bound for terminal cost}. The following is a restatement of \cite[Lem 1]{liu2024efficient}.
\begin{lemma} Let Assumption \ref{assumption: prelim: Mayer cost condition} hold. The function $\Bar{g}:\R^n\rightarrow \R$ defined in \eqref{eq: main: upper bound for g} is an upper bound for $g$ for all $x\in\R^n$.
\label{lemma: main: upper bound for g}
\end{lemma}
Consequently, selecting $\Bar{v}(T,x) = \Bar{g}(x)$ ensures that $\Bar{v}(T,x)$ is an upper bound for the terminal data $v(T,x) = g(x)$ as required by Theorem \ref{theorem: prelim: comparison result}. 
\begin{figure}[ht!]
    \centering
    \begin{subfigure}[t]{0.45\textwidth}
        \centering
        \includegraphics[width = \columnwidth]{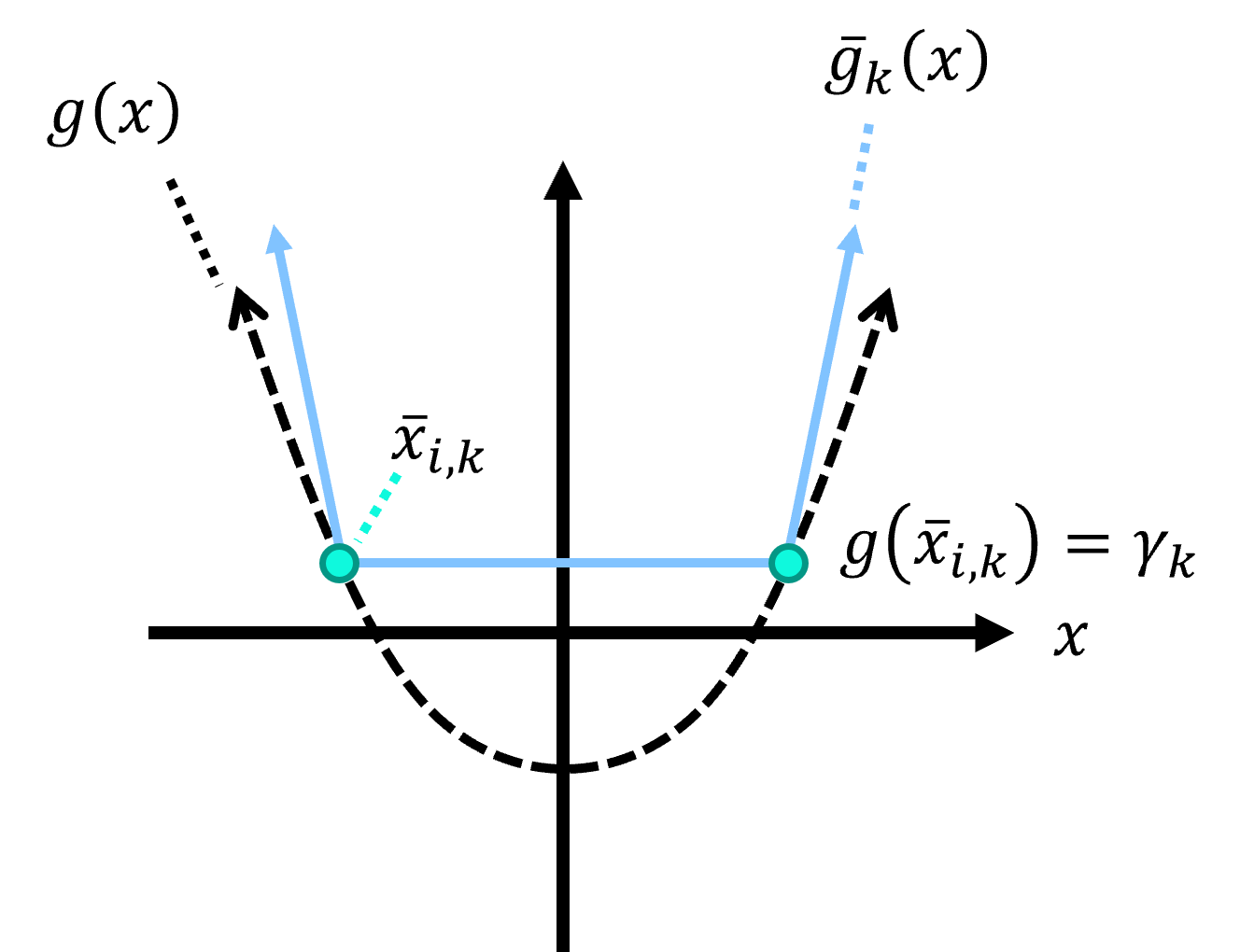}
        \caption{$\Bar{g}_k$}
        \label{subfig: main: single upper bound}
    \end{subfigure}%
    \hspace{0.3cm}
    \begin{subfigure}[t]{0.45\textwidth}
        \centering
        \includegraphics[width = \columnwidth]{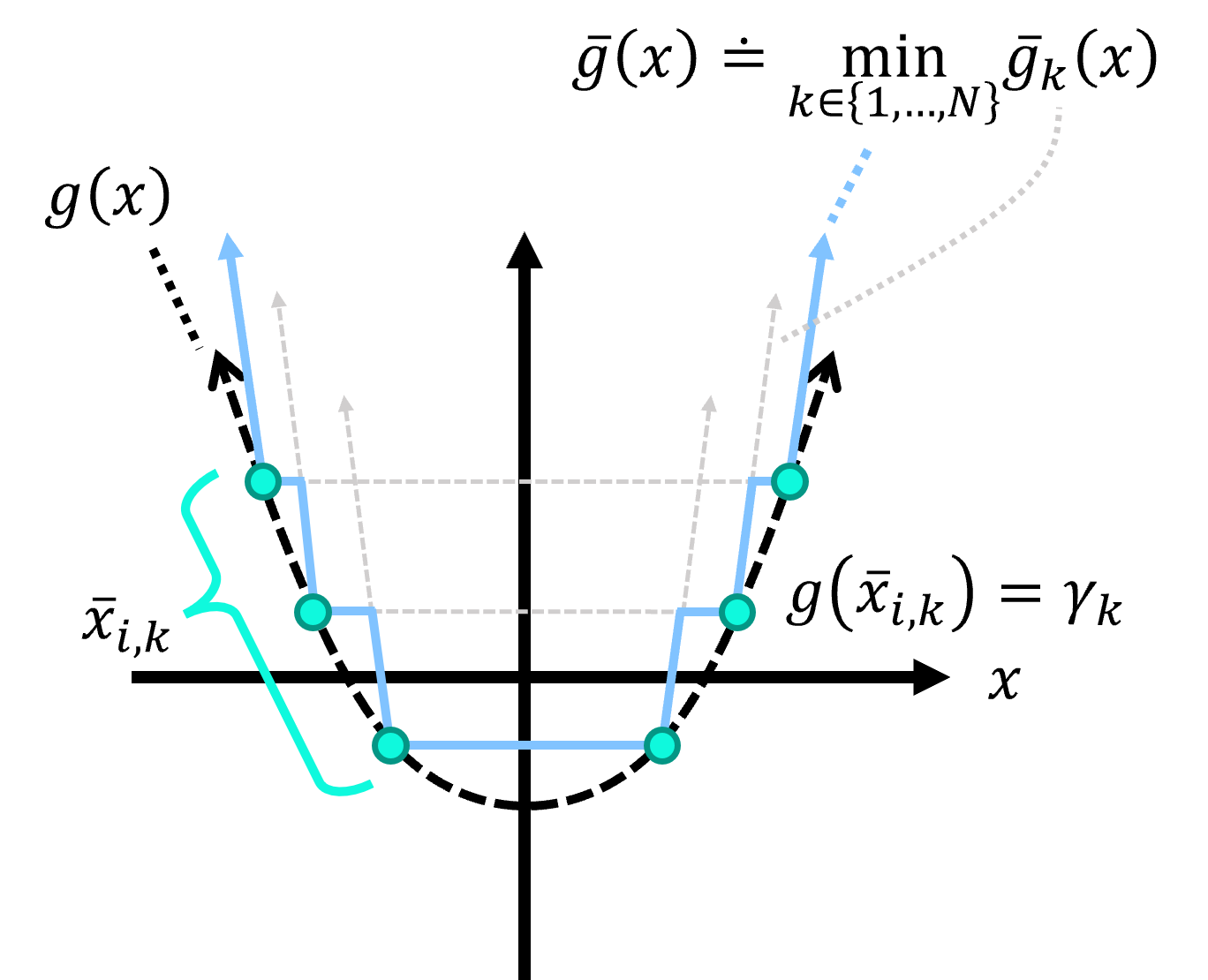}
        \caption{Minimum over $\Bar{g}_k$}
        \label{subfig: main: multiple upper bound}
    \end{subfigure}
    \caption{Illustration of the proposed upper bound $\Bar{g}$ in \eqref{eq: main: upper bound for g}, which is the minimum over the individual upper bounds $\Bar{g}_k$ in \eqref{eq: main: upper bound for g via each level set}. Points $\Bar{x}_{i,k}$, which characterise the upper bound, lie on the $\gamma_k$-level set of $g$.}
    \label{fig: main: upper bound for terminal cost}
\end{figure}

Next, consider the terminal value problem
\begin{equation}
    \begin{split}
        \dot{\xi}(s) &= A(s)\xi(s)+\omega(s), \quad \text{a.e. } s\in(t,T), \\
        \xi(T) &= x,
    \end{split}
    \label{eq: main: LTV system with trimmed input}
\end{equation}
in which $\xi(s)\in\R^n$ is the state and $\omega(s) \in \R^n$ is a `trimmed' control input both at time $s$. Assume that the control input $\omega$ is selected from the set
\begin{equation}
    \sW[t,T]   \doteq \left\{\omega:[t,T]\rightarrow \R^n\;| \;\omega \text{ measurable, } \omega(s)\in\W(s) \: \forall s\in[t,T] \right\},
    \label{eq: main: trimmed controls}
\end{equation}
in which the trimmed input set $\W(s)$ is defined in \eqref{eq: prelim: trimmed input set def}. Note that $\sW[t,T]$ is non-empty under Assumption \ref{assumption: prelim: trimmed input condition}, which roughly translates to requiring that player I has greater control authority than player II. Solutions of \eqref{eq: main: LTV system with trimmed input}, which are denoted by $\xi^{\omega}_{T,x}(s)$ at time $s$, will be used to parameterise the desired upper bound $\Bar{v}$. In particular, associate with each $k\in\{1,\ldots,N\}$ a collection of $n_k$ solutions of \eqref{eq: main: LTV system with trimmed input}
\begin{equation}
\xi^{\star}_{i,k}(s) \doteq \xi^{\omega_{i,k}}_{T,\Bar{x}_{i,k}}(s), \quad \forall s\in[t,T],
    \label{eq: main: collection of solutions}
\end{equation}
in which $\omega_{i,k}\in\sW[t,T]$ for all $i\in\{1,\ldots,n_k\}$ and the collection $\{\Bar{x}_{i,k}\}^{n_k}_{i=1}\subseteq g^{-1}\left(\{\gamma_k\}\right)$ defines the terminal condition of each solution $\xi^{\star}_{i,k}(\cdot)$ of \eqref{eq: main: LTV system with trimmed input}. Consider also a time-dependent distance function $d_A(\cdot,\cdot;\,\Omega):[0,T]\times\R^{n}\rightarrow \R$, which is defined by
\begin{equation}
    d_{A}(t,x;\,\Omega) \doteq \inf_{\xi \in \Omega}\norm{\Phi_A(T,t)(x - \xi)},
    \label{eq: main: time-dependent dist function def}
\end{equation}
in which $\Omega\subset \R^n$ is a given set and $\Phi_A(\cdot,\cdot)$ denotes the state-transition matrix associated with the dynamics $\dot{\xi}(s) = A(s)\xi(s)$. For each $k\in\{1,\ldots,N\}$, the $n_k$ solutions $\xi^{\star}_{i,k}(\cdot)$ of \eqref{eq: main: LTV system with trimmed input} may be used to construct a candidate upper bound $\Bar{v}_k:[0,T]\times\R^n\rightarrow \R$ by defining, 
\begin{equation}
    \begin{split}
     \Bar{v}_k(t,x) &\doteq L_g d_{A}(t,x;\,\Omega^\star_k(t)) + \gamma_k, \quad \forall (t,x)\in[0,T]\times\R^n,\\
    \Omega^\star_k(t) &\doteq \text{conv}\left( \{\xi^{\star}_{i,k}(t)\}^{n_k}_{i=1}\right).
    \end{split}
    \label{eq: main: upper-bounding surface kth level set}
\end{equation}
Then, the minimisation over each function $\Bar{v}_k$ defined by
\begin{equation}
    \Bar{v}(t,x) \doteq \min_{k\in\{1,\ldots,N\}}\Bar{v}_k(t,x),\:\, \forall (t,x)\in[0,T]\times\R^{n}.
    \label{eq: main: upper-bounding surface}
\end{equation}
yields another upper bound of $v$. The upper bound guarantees will follow from the comparison result in Theorem \ref{theorem: prelim: comparison result} by demonstrating that $\Bar{v}$ is a viscosity supersolution of \eqref{eq: prelim: HJI}.
\begin{lemma} Let Assumption \ref{assumption: prelim: trimmed input condition} hold. Then, $\Bar{v}$ as defined in \eqref{eq: main: upper-bounding surface} is a continuous viscosity supersolution of \eqref{eq: prelim: HJI}.
\label{lemma: main: viscosity supersolution via scaled and shifted distance functions}
\end{lemma}
\begin{proof}
First note that since $\U$ and $\D$ are closed, and $s\mapsto B(s)$ and $s\mapsto E(s)$ are continuous, by Assumption \ref{assumption: prelim: trimmed input condition}, the multifunction $s\mapsto \W(s)$ admits a measurable selection (see \cite[Thm A.5.2 and Prop A.5.3, p.292]{CS:04}). In particular, $\sW[t,T]$ is non-empty so the solutions $\xi^\star_{i,k}$ of \eqref{eq: main: LTV system with trimmed input}, which parameterise $\Bar{v}$, exist and $\Bar{v}$ is well-defined.

   \emph{Subclaim (a): For all $k\in\{1,\ldots,N\}$, $\Bar{v}_k$ as defined in \eqref{eq: main: upper-bounding surface kth level set} is a continuous viscosity supersolution of \eqref{eq: prelim: HJI}.} Continuity of $\Bar{v}_k$ is demonstrated in Lemma \ref{lemma: continuity of upper-bounding surface kth level set} in the appendix. Fix any $k\in\{1,\ldots,N\}$ and any $(t_0, x_0) \in (0,T)\times\R^{n}$ such that $D^-\Bar{v}_k(t_0,x_0)$ is non-empty. Let
    \begin{equation}
        \hat{x}_{0} \in \argmin_{\xi \in \Omega^{\star}_k(t_0)}\norm{\Phi_A(T,t_0)(x_0-\xi)}.
         \label{eq: main: viscosity supersolution proof 1}
    \end{equation}
    Note that the minimising set is non-empty due to continuity of $\xi \mapsto \norm{\Phi_A(T,t_0)(x_0-\xi)}$ and compactness of $\Omega^\star_k(t_0)$. By definition of $\Omega^\star_k(t_0)$ in \eqref{eq: main: upper-bounding surface kth level set},
    \begin{equation*}
        \hat{x}_{0} =  \sum^{n_k}_{i=1}\hat{\lambda}_i\xi^\star_{i,k}(t_0),\quad \{\hat{\lambda}_i\}^{n_k}_{i=1}\subset [0,1], \quad  \sum^{n_k}_{i=1}\hat{\lambda}_i=1.
%        \label{eq: main: viscosity supersolution proof 1}
    \end{equation*}
Let $\hat{\omega} \doteq \sum^{n_k}_{i=1}\hat{\lambda}_i\omega_{i,k}$ in which $\omega_{i,k}\in\sW[t_0,T]$ is the control associated with each solution $\xi^\star_{i,k}$. By convexity of $\W(t)$ (see Lemma \ref{lemma: convexity of trimmed input set} in the appendix), $\hat{\omega}(t) \in \W(t)$ for all $t\in[t_0,T]$. As any finite linear combination of measurable functions is measurable, $\hat{\omega} \in \sW[t_0,T]$. Let $\xi_0 \doteq \xi^{\hat{\omega}}_{t_0,x_0}$ and $\hat{\xi}_0\doteq\xi^{\hat{\omega}}_{t_0,\hat{x}_0}$ be solutions of \eqref{eq: main: LTV system with trimmed input} under the control $\hat{\omega}$ from $(t_0, x_0)$ and $(t_0, \hat{x}_0)$, respectively. Then, fix $t \in [t_0,T]$ to obtain
\begin{align*}
   \hat{\xi}_0(t) &= \Phi_A(t,t_0)\hat{x}_0 + \integral{t_0}{t}{\Phi_A(t,s)\hat{\omega}(s)}{s},\nn\\
   &=  \sum^{n_k}_{i=1}\hat{\lambda}_i\left(\Phi_A(t,t_0)\xi^\star_{i,k}(t_0) + \integral{t_0}{t}{\Phi_A(t,s)\omega_{i,k}(s)}{s}\right). \
\end{align*}
   Thus, $\hat{\xi}_0(t) = \sum^{n_k}_{i=1}\hat{\lambda}_i\xi^\star_{i,k}(t)$, which implies
    \begin{equation}
       \xi_0(t) - \sum^{n_k}_{i=1}\hat{\lambda}_i\xi^\star_{i,k}(t) = \xi_0(t)-\hat{\xi}_0(t)= \Phi_A(t,t_0)(x_0 - \hat{x}_0).\label{eq: main: viscosity supersolution proof 2}
   \end{equation}
   Selecting $\sum^{n_k}_{i=1}\hat{\lambda}_i\xi^{\star}_{i,k}(t)\in \Omega^\star_k(t)$ for $\xi$ in \eqref{eq: main: upper-bounding surface kth level set} followed by a substitution of \eqref{eq: main: viscosity supersolution proof 2} yields, 
   \begin{align}
       d_{A}\left(t,\xi_0(t);\,\Omega^{\star}_k(t)\right)&\leq \Big\lVert\Phi_A(T,t)\Big(\xi_0(t)-\sum^{n_k}_{i=1}\hat{\lambda}_i\xi^{\star}_{i,k}(t)\Big)\Big\rVert  \label{eq: main: viscosity supersolution proof 3} \\
       &= \big\lVert\Phi_A(T,t_0)(x_0 -\hat{x}_0)\big\rVert = d_{A}(t_0,x_0;\,\Omega^{\star}_k(t_0)), 
   \end{align}
   in which the last equality follows by \eqref{eq: main: viscosity supersolution proof 1}. Hence, by \eqref{eq: main: viscosity supersolution proof 3} and the definition of $\Bar{v}_k$ in \eqref{eq: main: upper-bounding surface kth level set}, 
   \begin{equation}
       \Bar{v}_k(t,\xi_0(t)) - \Bar{v}_k(t_0,x_0)= L_g[d_{A}\left(t,\xi_0(t);\,\Omega^{\star}_k(t)\right) - d_{A}(t_0,x_0;\,\Omega^{\star}_k(t_0))] \leq 0, \quad\forall t\in[t_0,T]. \label{eq: main: viscosity supersolution proof 4}
   \end{equation}

    Next, fix any $(q, p)\in D^{-}\Bar{v}_k(t_0,x_0)$ and define $h:[0,T]\times\R^n\rightarrow \R$ by
    \begin{equation*}
        h(t,x) \doteq \Bar{v}_k(t,x)-\Bar{v}_k(t_0,x_0)-\iprod{(q,p)}{(t,x)-(t_0,x_0)}.
    \end{equation*}
    From Definition \ref{def: prelim: reg subdifferentials and superdifferentials},
\begin{equation}
\liminf_{(t,x)\rightarrow(t_0,x_0)}\frac{h(t,x)}{\norm{(t,x)-(t_0,x_0)}}  \geq 0.
    \label{eq: main: viscosity supersolution proof 5}
\end{equation}
Taking the limit in \eqref{eq: main: viscosity supersolution proof 5} along $\xi_0(\cdot)$ and noting that $\xi_0(t_0) = x_0$ yields 
\begin{equation}
        \liminf_{t\rightarrow t_0}\frac{h(t,\xi_0(t))}{\lVert(t,\xi_0(t))-(t_0,x_0)\rVert}\geq  \liminf_{(t,x)\rightarrow (t_0,x_0)} \frac{h(t,x)}{\norm{(t,x)-(t_0,x_0)}}.  \label{eq: main: viscosity supersolution proof 6}
\end{equation}\normalsize
By similar arguments to \cite[Lem 7.1.2(i), p.187]{CS:04}, given $r > 0$, there exists $R> 0$ such that $\big\lVert A(t)\xi_0(t)+\hat{\omega}(t)\big\rVert \leq R$ for all $t\in\B_r(t_0)\cap [0,T]$. Thus, $\lVert \xi_0(t) - x_0 \rVert \leq R |t - t_0|$, which implies 
\begin{equation}
    |t - t_0| \leq \big\lVert(t,\xi_0(t))-(t_0,x_0)\big\rVert \leq \sqrt{1+R^2}|t - t_0|, \label{eq: main: viscosity supersolution proof 7}
\end{equation}
for all $t\in \B_r(t_0)\cap[0,T]$. 
Since $\eta(\epsilon)\doteq \inf_{t\in\B_{\epsilon}(t_0)\backslash\{t_0\}}\frac{h(t,\xi_0(t))}{\lVert (t,\xi_0(t)) - (t_0,x_0)\rVert}$ is non-decreasing as $\epsilon \rightarrow 0^+$, from \eqref{eq: main: viscosity supersolution proof 5}-\eqref{eq: main: viscosity supersolution proof 6}, either: (i) $\eta(\epsilon) > 0$ for sufficiently small $\epsilon > 0$ or (ii) $\eta(\epsilon) \leq 0$ for all $\epsilon > 0$. In case (i), $\inf_{t\in\B_{\epsilon}(t_0)\backslash\{t_0\}}\frac{h(t,\xi_0(t))}{|t-t_0|} \geq \eta(\epsilon)$ for sufficiently small $\epsilon$, which follows from the left-hand inequality in \eqref{eq: main: viscosity supersolution proof 7}. In case (ii), $\inf_{t\in\B_{\epsilon}(t_0)\backslash\{t_0\}}\frac{h(t,\xi_0(t))}{\sqrt{1+R^2}|t-t_0|} \geq \eta(\epsilon)$ for all sufficiently small $\epsilon$, which follows from the right-hand inequality in \eqref{eq: main: viscosity supersolution proof 7}. Thus,
\begin{equation}
    \liminf_{t\rightarrow t_0}\frac{h(t,\xi_0(t))}{|t-t_0|} \geq \kappa \lim_{\epsilon \rightarrow 0^+}\eta(\epsilon) = \kappa \liminf_{t\rightarrow t_0}\frac{h(t,\xi_0(t))}{\lVert(t,\xi_0(t))-(t_0,x_0)\rVert},
        \label{eq: main: viscosity supersolution proof 8}
\end{equation}
in which $\kappa = 1$ for case (i) or $\kappa = \sqrt{1+R^2}$ for case (ii). For either case, by \eqref{eq: main: viscosity supersolution proof 5}, \eqref{eq: main: viscosity supersolution proof 6}, and \eqref{eq: main: viscosity supersolution proof 8}, 
\begin{equation}
\liminf_{t\rightarrow t_0^+}\frac{h(t,\xi_0(t))}{(t-t_0)} \geq  \liminf_{t\rightarrow t_0}\frac{h(t,\xi_0(t))}{|t-t_0|} \geq 0,
    \label{eq: main: viscosity supersolution proof 9}
\end{equation}
in which the left inequality in \eqref{eq: main: viscosity supersolution proof 9} follows by taking the infimum over a smaller set. Expanding $h$ in the left side of \eqref{eq: main: viscosity supersolution proof 9} and using \eqref{eq: main: viscosity supersolution proof 4} yields
\begin{equation}
    -q + \liminf_{t\rightarrow t_0^+}\Big\langle -p, \textstyle{\frac{\xi_0(t) -x_0}{t-t_0}}\Big\rangle\geq 0.
    \label{eq: main: viscosity supersolution proof 10}
\end{equation}
Next, given any $\epsilon >0$, there exists $\delta_1 > 0$ such that, $A(t)\xi_0(t) \subseteq A(t_0)x_0 \oplus \B_{\epsilon \slash 2}(0)$ whenever $|t-t_0| < \delta_1$, which follows by continuity of $A$ and solutions of \eqref{eq: main: LTV system with trimmed input}. Similarly, continuity of the set-valued map $t\mapsto\W(t)$ (see Lemma \ref{lemma: continuity of trimmed control set} in the appendix) implies that there exists $\delta_2>0$ such that $\W(t)\subseteq \W(t_0)\oplus\B_{\epsilon\slash 2}(0)$ for all $|t-t_0| < \delta_2$. Thus, 
\begin{equation}
    \textstyle{\frac{1}{t- t_0}}\integral{t_0}{t}{A(s)\xi_0(s)+\hat{\omega}(s)}{s} \in \left\{A(t_0)x_0\right\}\oplus \W(t_0)\oplus\B_{\epsilon}(0),
    \label{eq: main: viscosity supersolution proof 11}
\end{equation}
whenever $|t-t_0| < \min(\delta_1, \delta_2)$, which follows by convexity of the set on the right side of \eqref{eq: main: viscosity supersolution proof 11} and \cite[Thm A.127, p.282]{CS:04}. 
By \eqref{eq: main: LTV system with trimmed input}, \eqref{eq: main: viscosity supersolution proof 11}, and the fundamental theorem of calculus,
\begin{equation}
    \liminf_{t\rightarrow t_0^+} \frac{\xi_0(t) - x_0}{t- t_0} \in \left\{A(t_0)x_0\right\}\oplus \W(t_0) .
     \label{eq: main: viscosity supersolution proof 12}
\end{equation}
Substituting \eqref{eq: main: viscosity supersolution proof 12} into \eqref{eq: main: viscosity supersolution proof 10}, for some $\Tilde{\omega}\in\W(t_0)$, 
\begin{equation}
-q + \langle -p, A(t_0)x_0 + \Tilde{\omega}\rangle \geq 0.
\label{eq: main: viscosity supersolution proof 13}
\end{equation}
Let $\Tilde{d}\in \argmin_{d\in\D}\langle-p,E(t_0)d \rangle$. By definition of the Minkowski difference, there exists $\Tilde{u} \in \U$ such that $\Tilde{\omega}-E(t_0)\Tilde{d} = B(t_0)\Tilde{u}$. Thus,
\begin{align*}
   -q+H(t_0,x_0,p) &= -q + \max_{u\in\U}\langle-p, A(t_0)x_0 + B(t_0)u\rangle +  \min_{d\in\D}\langle-p,E(t_0)d\rangle \nn\\
   &\geq -q+\langle -p, A(t_0)x_0 + B(t_0)\Tilde{u}\rangle + \min_{d\in\D}\langle-p,E(t_0)d \rangle\nn \\
   &=-q + \langle -p, A(t_0)x_0 + \Tilde{\omega} -E(t_0)\Tilde{d} + E(t_0)\Tilde{d}\rangle \geq 0,
\end{align*}
in which the last inequality follows by \eqref{eq: main: viscosity supersolution proof 13}. As $k\in\{1,\ldots,N\}$, $(t_0,x_0)\in(0,T)\times\R^{n}$, and $(q,p)\in D^{-}\Bar{v}_k(t_0,x_0)$ are arbitrary, by Definition \ref{def: prelim: visc solution}, subclaim (a) holds.

Finally, the minimisation over any finite collection of continuous viscosity supersolutions of \eqref{eq: prelim: HJI} yields another continuous viscosity supersolution (apply induction to \cite[Prop 2.1, p.34]{bardi1997optimal}). Thus, $\Bar{v}$ in \eqref{eq: main: upper-bounding surface} is a continuous viscosity supersolution of \eqref{eq: prelim: HJI}.
\end{proof}

\begin{theorem} Let Assumptions \ref{assumption: prelim: Mayer cost condition}-\ref{assumption: prelim: trimmed input condition} hold. Then, $\Bar{v}$ as defined in \eqref{eq: main: upper-bounding surface} is an upper bound for the value function $v$ in \eqref{eq: prelim: value function} for all $(t,x)\in[0,T]\times\R^n$. 
    \label{theorem: main: upper-bounding the value function}
\end{theorem}
\begin{proof}
    From \eqref{eq: main: time-dependent dist function def}, \eqref{eq: main: upper-bounding surface kth level set}, \eqref{eq: main: upper bound for g via each level set}, and \eqref{eq: main: convex hull of terminal states}, for any $k\in\{1,\ldots,N\}$,
    \begin{align}
        \Bar{v}_k(T,x) &= L_g d_{A}(T,x;\,\Omega^\star_k(T)) + \gamma_k = L_g d(x;\,\Bar{\Omega}_k) + \gamma_ k =\Bar{g}_k(x).\label{eq: main: upper bound for value function proof 1}
    \end{align}
    Taking the minimum of \eqref{eq: main: upper bound for value function proof 1} over $k\in\{1,\ldots,N\}$ and using Lemma \ref{lemma: main: upper bound for g}, $\Bar{v}(T,x) = \Bar{g}(x) \geq g(x)$ for all $x\in\R^n$. Noting that $v$ in \eqref{eq: prelim: value function} is a viscosity solution of \eqref{eq: prelim: HJI} (Theorem \ref{theorem: prelim: value function and HJI equation}) and hence a viscosity subsolution of \eqref{eq: prelim: HJI} (Definition \ref{def: prelim: visc solution}), by Lemma \ref{lemma: main: viscosity supersolution via scaled and shifted distance functions} and Theorem \ref{theorem: prelim: comparison result}, $\Bar{v}(t,x) \geq v(t,x)$ for all $(t,x)\in[0,T]\times\R^n$.
\end{proof}

The solutions $\{\xi^{\star}_{i,k}(\cdot)\}^{n_k,N}_{i,k=1}$ parameterise the proposed upper bound $\Bar{v}$, and hence determine its degree of conservatism. By definition of $\Bar{v}_k$ in \eqref{eq: main: upper-bounding surface kth level set}, it can be shown that each solution $\xi^\star_{i,k}$ remains on the $\gamma_k$-level set of $\Bar{v}_k$. If these solutions are selected such that they also lie on the $\gamma_k$-level set of $v$ as they evolve backwards in time, then the upper bound will coincide with the value function, yielding a tight upper bound. To this end, consider the terminal value problem 
\begin{gather}
\left\{\;
\begin{split}
    &\Dot{\xi}^\star(s) = A(s)\xi^\star(s) + \omega^\star(s), \quad \Dot{\lambda}(s) = -A'(s)\lambda(s), \quad \text{a.e. } s\in(t,T), \\
    &\omega^\star(s) \in \argmax_{\omega\in\W(s)}\iprod{-\lambda(s)}{\omega},\\
     &\xi^\star(T) = \Bar{x}, \quad \lambda(T) \in D^{-}g(\Bar{x}).
    \end{split}\right. 
    \label{eq: main: PMP trimmed input dynamics}
\end{gather}
\begin{remark}
    The subdifferential of $g$ at $\Bar{x}$, i.e., $D^-g(\Bar{x})$, is defined analogously via Definition \ref{def: prelim: reg subdifferentials and superdifferentials} with $x\in\R^n$ replacing $(t,x)\in (0,T)\times\R^n$ and $p\in\R^n$ replacing $(q, p)\in\R\times\R^n$. As $g$ is convex (by Assumption \ref{assumption: prelim: Mayer cost condition}), $D^-g(\Bar{x})$ is non-empty for all $\Bar{x}\in\R^n$ (see, e.g., \cite[Prop 8.12 and Thm 8.13, pp.308-309]{rockafellar2009variational}).
    \label{remark: main: subdifferential of g}
\end{remark}
Before demonstrating that \eqref{eq: main: PMP trimmed input dynamics} generates the desired set of solutions, it will be convenient to first demonstrate that under Assumptions \ref{assumption: prelim: trimmed input condition}-\ref{assumption: prelim: recoverability of player 1 set} the value function $v$ coincides with a corresponding value function $w:[0,T]\times\R^n\rightarrow\R$ defined for \eqref{eq: main: LTV system with trimmed input} via
\begin{equation}
    w(t,x)\doteq \inf_{\omega\in\sW[t,T]}g\left(\xi^{\omega}_{t,x}(T)\right), \quad \forall (t,x)\in[0,T]\times\R^n.
    \label{eq: main: value function for trimmed input system}
\end{equation}

\begin{lemma} Let Assumptions \ref{assumption: prelim: trimmed input condition}-\ref{assumption: prelim: recoverability of player 1 set} hold and consider $v$ and $w$ as defined in \eqref{eq: prelim: value function} and \eqref{eq: main: value function for trimmed input system}, respectively. Then,
\begin{enumerate}[(a)]
    \item \label{item: equivalence of value functions 1} $w\in \sC((0,T)\times\R^n;\,\R)$ is a viscosity solution of 
    \begin{equation}
-w_t(t,x)+\cH(t,x,\nabla w(t,x)) = 0, \quad \forall(t,x)\in(0,T)\times\R^n,
    \label{eq: main: equivalence of value functions proof 1}
\end{equation}
in which $\cH:[0,T]\times\R^n\times\R^n\rightarrow\R$ is defined by
\begin{equation}
    \cH(t,x,p) \doteq \max_{\omega\in \W(t)}\iprod{-p}{A(t)x+\omega};
    \label{eq: main: equivalence of value functions proof 2}
\end{equation}
\item \label{item: equivalence of value functions 2} $H(t,x,p) = \cH(t,x,p)$ for all $(t,x,p)\in[0,T]\times\R^n\times\R^n$; and 
\item \label{item: equivalence of value functions 3} $v(t,x) = w(t,x)$ for all $(t,x)\in[0,T]\times\R^n$.
\end{enumerate}
    \label{lemma: main: equivalence of value functions}
\end{lemma}
\begin{proof}
\underline{Assertion (\ref{item: equivalence of value functions 1}).} The proof proceeds similarly to \cite[Thm 7.2.4, p.194]{CS:04} so some steps will be omitted.  Using $\sW[t,T]$ as the set of admissible controls, the same arguments as \cite[Thm 7.2.2, p.192]{CS:04} can be used to demonstrate that $w$ satisfies the dynamic programming principle
\begin{equation}
    w(t,x) = \inf_{\omega\in\sW[t,s]}w\left(s,\xi^{\omega}_{t,x}(s)\right), \quad \forall t\in[0,T],\: s\in[t,T],\: x\in\R^n.
    \label{eq: main: equivalence of value functions proof 3}
\end{equation}
Moreover, by Assumption \ref{assumption: prelim: Mayer cost condition}, the same arguments as \cite[Thm 7.2.3, p.193]{CS:04} may be used to demonstrate continuity of $w$.

Fix $(t_0,x_0)\in(0,T)\times\R^n$, $(q,p)\in D^-w(t_0,x_0)$, and $\Tilde{\omega}\in\sW[t_0,T]$. By the same arguments as \eqref{eq: main: viscosity supersolution proof 12},
\begin{equation}
    \liminf_{t\rightarrow t_0^+} \frac{\xi^{\Tilde{\omega}}_{t_0,x_0}(t) - x_0}{t - t_0} \in \{A(t_0)x_0\}\oplus \W(t_0).
\end{equation}
Thus, there exists $\Bar{\omega}\in\W(t_0)$ such that
\begin{equation}
-q + \cH(t_0,x_0,p) \geq -q + \big\langle -p,  A(t_0)x_0+\Bar{\omega}\big\rangle = -q+\liminf_{t\rightarrow t_0^+} \Big\langle -p, \tfrac{\xi^{\Tilde{\omega}}_{t_0,x_0}(t) - x_0}{t - t_0}\Big\rangle . \label{eq: main: equivalence of value functions proof 4}
\end{equation}
As $\Tilde{\omega}\in\sW[t_0,T]$ is arbitrary, by \eqref{eq: main: equivalence of value functions proof 3}, it is possible to select $\Tilde{\omega}$ such that $ w\big(t,\xi^{\Tilde{\omega}}_{t_0,x_0}(t)\big) - w(t_0,x_0)\leq (t - t_0)^2$ for all $t\in[t_0,T]$. From \eqref{eq: main: equivalence of value functions proof 4},
\begin{align}
    -q + \cH(t_0,x_0,p) &\geq \liminf_{t\rightarrow t_0^+} \Big\{t-t_0 + \Big\langle -(q,p), \tfrac{(t - t_0,\xi^{\Tilde{\omega}}_{t_0,x_0}(t) - x_0)}{t - t_0}\Big\rangle\Big\}\nn \\
    &\geq \liminf_{t\rightarrow t_0^+} \frac{w\big(t,\xi^{\Tilde{\omega}}_{t_0,x_0}(t)\big) - w(t_0,x_0) -\langle(q,p),(t-t_0,\xi^{\Tilde{\omega}}_{t_0,x_0}(t) - x_0)\rangle}{t-t_0}.\label{eq: main: equivalence of value functions proof 5}
\end{align}
The same arguments from \eqref{eq: main: viscosity supersolution proof 5}-\eqref{eq: main: viscosity supersolution proof 9} hold by replacing $\Bar{v}_k$ in \eqref{eq: main: viscosity supersolution proof 5} with $w$ and replacing $\xi_0(t)$ in \eqref{eq: main: viscosity supersolution proof 6}-\eqref{eq: main: viscosity supersolution proof 9} with $\xi^{\Tilde{\omega}}_{t_0,x_0}(t)$. By \eqref{eq: main: viscosity supersolution proof 9} and \eqref{eq: main: equivalence of value functions proof 5},
\begin{equation}
    -q + \cH(t_0,x_0,p) \geq 0,
    \label{eq: main: equivalence of value functions proof 6}
\end{equation}
which by Definition \ref{def: prelim: visc solution} implies that $w$ is a viscosity supersolution of \eqref{eq: main: equivalence of value functions proof 1}.

Now, fix $(q,p) \in D^+w(t_0,x_0)$ and let $\omega^\star \in \sW[t_0,T]$ be a control satisfying
\begin{equation}
    \omega^\star(s) \in \argmax_{\omega\in\W(s)}\langle-p,\omega\rangle, \quad \forall s\in[t_0,T].
    \label{eq: main: equivalence of value functions proof 7}
\end{equation}
Note that by Assumption \ref{assumption: prelim: trimmed input condition}, $\W(s)$ is non-empty and compact for all $s\in[t_0,T]$, and by continuity of $B$ and $E$, $\text{cl}\big(\bigcup_{s\in[t_0,T]}\W(s)\big)\subset\R^n$ is also compact and thus separable. As $(s,\omega)\mapsto \langle-p,\omega\rangle$ is continuous, by \cite[Thm 18.19, p.605]{AB2006infinite}, there exists a measurable selection $s\mapsto \omega^\star(s)$ on $[t_0,T]$ satisfying \eqref{eq: main: equivalence of value functions proof 7}. Next, by continuity of $A$, continuity of solutions of \eqref{eq: main: LTV system with trimmed input}, Lemma \ref{lemma: continuity of trimmed control set}, and Theorem \ref{theorem: continuity of supremal value functions}, the map $t\mapsto \cH\big(t,\xi^{\omega^\star}_{t_0,x_0}(t), p\big)$ is continuous. Thus, from \eqref{eq: main: equivalence of value functions proof 7},
\begin{align}
    \lim_{t\rightarrow t_0^+} \Big\langle -p, \tfrac{\xi^{\omega^\star}_{t_0,x_0}(t) - x_0}{t - t_0}\Big\rangle &=     \lim_{t\rightarrow t_0^+} \frac{1}{t-t_0} \Big\langle -p, \integral{t_0}{t}{A(s)\xi^{\omega^\star}_{t_0,x_0}(s)+\omega^\star(s)}{s}\Big\rangle \nn \\
    &= \lim_{t\rightarrow t_0^+} \frac{1}{t-t_0} \integral{t_0}{t}{\big\langle -p, A(s)\xi^{\omega^\star}_{t_0,x_0}(s)+\omega^\star(s)\big\rangle}{s} \nn \\
     &=     \lim_{t\rightarrow t_0^+} \frac{1}{t-t_0} \integral{t_0}{t}{\cH\big(s,\xi^{\omega^\star}_{t_0,x_0}(s),p\big)}{s}=\cH(t_0,x_0,p). \label{eq: main: equivalence of value functions proof 8}
\end{align}
By Definition \ref{def: prelim: reg subdifferentials and superdifferentials},
\begin{align}
    &\limsup_{t\rightarrow t_0^+}\frac{w\big(t,\xi^{\omega^\star}_{t_0,x_0}(t)\big)-w(t_0,x_0)-\iprod{(q,p)}{\big(t,\xi^{\omega^\star}_{t_0,x_0}(t)\big)-(t_0,x_0)}}{\norm{(t,\xi^{\omega^\star}_{t_0,x_0}(t))-(t_0,x_0)}}  \nn \\
    &\qquad \qquad \qquad \leq \limsup_{(t,x)\rightarrow(t_0,x_0)}\frac{w(t,x)-w(t_0,x_0)-\iprod{(q,p)}{(t,x)-(t_0,x_0)}}{\norm{(t,x)-(t_0,x_0)}}\leq 0 .\label{eq: main: equivalence of value functions proof 9}
\end{align}
By \eqref{eq: main: equivalence of value functions proof 3}, $w\big(t,\xi^{\omega^\star}_{t_0,x_0}(t)\big)-w(t_0,x_0) \geq 0$, so from \eqref{eq: main: equivalence of value functions proof 9},
\begin{equation}
     \limsup_{t\rightarrow t_0^+}\frac{-\iprod{(q,p)}{\big(t,\xi^{\omega^\star}_{t_0,x_0}(t)\big)-(t_0,x_0)}}{\norm{(t,\xi^{\omega^\star}_{t_0,x_0}(t))-(t_0,x_0)}}\leq 0 \label{eq: main: equivalence of value functions proof 10}.
\end{equation}
By \eqref{eq: main: viscosity supersolution proof 7} and \eqref{eq: main: equivalence of value functions proof 10} and using analogous arguments to \eqref{eq: main: viscosity supersolution proof 5}-\eqref{eq: main: viscosity supersolution proof 9},
\begin{equation}
    \limsup_{t\rightarrow t_0^+}\frac{-\iprod{(q,p)}{\big(t,\xi^{\omega^\star}_{t_0,x_0}(t)\big)-(t_0,x_0)}}{(t-t_0)}\leq 0. \label{eq: main: equivalence of value functions proof 11}
\end{equation}
Finally, by \eqref{eq: main: equivalence of value functions proof 11} and \eqref{eq: main: equivalence of value functions proof 8},
\begin{equation*}
    -q + \cH(t_0,x_0,p) \leq 0,
\end{equation*}
which implies that $w$ is a viscosity subsolution of \eqref{eq: main: equivalence of value functions proof 1}. Consequently, assertion (\ref{item: equivalence of value functions 1}) holds.

\underline{Assertion (\ref{item: equivalence of value functions 2}).} Fix $(t,x,p)\in [0,T]\times\R^n\times\R^n$ and let 
\begin{equation*}
    \omega^\star \in \argmax_{\omega\in\W(t)}\langle-p,\omega\rangle, \quad u^\star \in \argmax_{u\in\U}\langle-p,B(t)u\rangle, \quad d^\star \in \argmin_{d\in\D}\langle -p,E(t)d\rangle.
\end{equation*}
Since $\omega^\star - E(t)d^\star = B(t)\Tilde{u}$ for some $\Tilde{u}\in\U$, 
\begin{align}
    \cH(t,x,p) &=\langle -p, A(t)x+\omega^\star - E(t)d^\star + E(t)d^\star\rangle \nn \\
    &= \min_{d\in\D}\langle -p, A(t)x+B(t)\Tilde{u} + E(t)d\rangle \nn \\
    &\leq \max_{u\in\U}\min_{d\in\D}\langle -p, A(t)x+B(t)u + E(t)d\rangle = H(t,x,p).
    \label{eq: main: equivalence of value functions proof 12}
\end{align}
By Assumption \ref{assumption: prelim: recoverability of player 1 set}, $B(t)u^\star = \Tilde{\omega} -E(t)\Tilde{d}$ for some $\Tilde{\omega}\in\W(t)$ and $\Tilde{d}\in\D$. Since $d^\star$ minimises $\langle-p,E(t)d\rangle$, $\langle -p, E(t)d^\star\rangle \leq \langle -p,E(t)\Tilde{d}\rangle$. Consequently,
\begin{align}
    H(t,x,p) &= \langle -p, A(t)x+B(t)u^\star + E(t)d^\star\rangle\nn \\
    &=\langle -p, A(t)x+\Tilde{\omega} -E(t)\Tilde{d} + E(t)d^\star \rangle \nn \\
    &\leq \langle -p, A(t)x+\Tilde{\omega}\rangle \leq \max_{\omega\in\W(t)}\langle -p, A(t)x+\omega\rangle = \cH(t,x,p).
    \label{eq: main: equivalence of value functions proof 13}
\end{align}
By \eqref{eq: main: equivalence of value functions proof 12} and \eqref{eq: main: equivalence of value functions proof 13}, assertion (\ref{item: equivalence of value functions 2}) holds.

\underline{Assertion (\ref{item: equivalence of value functions 3}).} Noting that $w(T,x) = g(x)$ for all $x\in \R^n$, by assertions (\ref{item: equivalence of value functions 1}) and (\ref{item: equivalence of value functions 2}), $w$ is a viscosity solution to \eqref{eq: prelim: HJI with terminal condition}. Since \eqref{eq: prelim: HJI with terminal condition} admits the unique solution $v$ on $[0,T]\times\R^n$ (Theorem \ref{theorem: prelim: value function and HJI equation}), assertion (\ref{item: equivalence of value functions 3}) holds.
\end{proof}

\begin{corollary} Let Assumptions \ref{assumption: prelim: Mayer cost condition}-\ref{assumption: prelim: recoverability of player 1 set} hold and consider $v$ and $\Bar{v}$ as defined in \eqref{eq: prelim: value function} and \eqref{eq: main: upper-bounding surface}, respectively. Given $k\in\{1,\ldots,N\}$ and $i\in\{1,\ldots,n_k\}$, if $\xi^\star_{i,k}$, which characterises $\Bar{v}$, and its corresponding control $\omega_{i,k}$ satisfies the terminal value problem in \eqref{eq: main: PMP trimmed input dynamics} for any costate solution $\lambda_{i,k}$, then,
    \begin{equation}
    v\big(t,\xi^\star_{i,k}(t)\big) = \Bar{v}\big(t,\xi^\star_{i,k}(t)\big) = \gamma_k.
        \label{eq: main: upper bound is tight with exact value}
    \end{equation}
    \label{corollary: main: tight upper bound}
\end{corollary}
\begin{proof}
Let the tuple $(\xi^\star_{i,k},\lambda_{i,k},\omega_{i,k})$ satisfy \eqref{eq: main: PMP trimmed input dynamics}, in which $\xi^\star_{i,k}$ characterises $\Bar{v}$ in \eqref{eq: main: upper-bounding surface}, $\omega_{i,k}$ is its corresponding control, and $\lambda_{i,k}(s) = \Phi_{-A'}(s,T)\lambda_{i,k}(T)$ for all $s\in[t,T]$ is the corresponding costate solution with $\lambda_{i,k}(T)\in D^-g\big(\xi^\star_{i,k}(T)\big)$ fixed.
    
    \emph{Subclaim (a): $\Bar{v}\big(t,\xi^\star_{i,k}(t)\big) \leq \gamma_k$.} Consider $\Bar{v}_k$ as defined in \eqref{eq: main: upper-bounding surface kth level set}. Since $\xi^\star_{i,k}(t) \in \Omega^\star_k(t)$, 
    \begin{equation*}
        0 \leq \inf_{\xi \in \Omega^\star_k(t)}\norm{\Phi_A(T,t)(\xi^\star_{i,k}(t)-\xi)} \leq \norm{\Phi_A(T,t)(\xi^\star_{i,k}(t)-\xi^\star_{i,k}(t))} = 0.
    \end{equation*}
    Thus, $\Bar{v}_k\big(t,\xi^\star_{i,k}(t)\big) = \gamma_k$, which by \eqref{eq: main: upper-bounding surface} yields subclaim (a). 

    \emph{Subclaim (b): $v\big(t,\xi^\star_{i,k}(t)\big) = \gamma_k$.} Fix any $\hat{\omega}\in\sW[t,T]$ and let $\hat{\xi}(s) \doteq \xi^{\hat{\omega}}_{t,\xi^\star_{i,k}(t)}(s)$ for all $s\in[t,T]$. As $g$ is convex and $\lambda_{i,k}(T) \in D^-g(\xi^\star_{i,k}(T))$, by \cite[Prop 8.12, p.308]{rockafellar2009variational},
    \begin{equation}
        g(\hat{\xi}(T)) \geq g(\xi^\star_{i,k}(T)) + \langle\lambda_{i,k}(T),\hat{\xi}(T)-\xi^\star_{i,k}(T)\rangle.
       \label{eq: main: tight upper bound proof 1}
    \end{equation}
   Considering the second term on the right-hand side,
    \begin{align}
        \langle\lambda_{i,k}(T),\,\hat{\xi}(T)-\xi^\star_{i,k}(T)\rangle &= \Big\langle \lambda_{i,k}(T), \integral{t}{T}{\Phi_{A}(T, s)(\hat{\omega}(s)-\omega_{i,k}(s))}{s} \Big\rangle  \nn \\
        &= \integral{t}{T}{\langle -\Phi'_A(T,s)\lambda_{i,k}(T),\omega_{i,k}(s) -\hat{\omega}(s)\rangle}{s}\nn \\
        &= \integral{t}{T}{\langle -\Phi_{-A'}(s,T)\lambda_{i,k}(T),\omega_{i,k}(s) -\hat{\omega}(s)\rangle}{s} \nn\\
                &= \integral{t}{T}{\langle -\lambda_{i,k}(s),\omega_{i,k}(s) -\hat{\omega}(s)\rangle}{s}.      \label{eq: main: tight upper bound proof 2}
    \end{align}
    Since $\omega_{i,k}(s)$ maximises $\langle -\lambda_{i,k}(s),\,\omega\rangle$ on $\W(s)$, $\langle -\lambda_{i,k}(s),\omega_{i,k}(s) \rangle \geq \langle -\lambda_{i,k}(s),\hat{\omega}(s) \rangle$ for all $s\in[t,T]$. Consequently, by \eqref{eq: main: tight upper bound proof 1} and \eqref{eq: main: tight upper bound proof 2},
    \begin{equation}
        g(\hat{\xi}(T)) \geq g\left(\xi^\star_{i,k}(T)\right).
        \label{eq: main: tight upper bound proof 3}
    \end{equation}
    As $\hat{\omega}\in\sW[t,T]$ is arbitrary and $\xi^\star_{i,k}(T)=\Bar{x}_{i,k} \in g^{-1}\left(\{\gamma_k\}\right)$, taking the $\inf$ of \eqref{eq: main: tight upper bound proof 3} over $\sW[t,T]$ yields
    \begin{equation}
    w\big(t,\xi^\star_{i,k}(t)\big) = \inf_{\omega \in \sW[t,T]}g\big(\xi^{\omega}_{t,\xi^\star_{i,k}(t)}(T)\big) \geq g\big(\xi^\star_{i,k}(T)\big) = \gamma_k,
    \label{eq: main: tight upper bound proof 4}
    \end{equation} 
    in which $w$ is defined in \eqref{eq: main: value function for trimmed input system}. Since $\omega_{i,k}\in\sW[t,T]$, \eqref{eq: main: tight upper bound proof 4} implies $w\big(t,\xi^\star_{i,k}(t)\big) =\gamma_k$. By Lemma \ref{lemma: main: equivalence of value functions}(\ref{item: equivalence of value functions 3}), subclaim (b) holds. Then, by Theorem \ref{theorem: main: upper-bounding the value function} and subclaim (a), \eqref{eq: main: upper bound is tight with exact value} holds. 
\end{proof}

\subsection{Lower-bounding the Value Function}
\label{sec: lower-bounding the value function}
A lower bound for the value $v$ in \eqref{eq: prelim: value function} can be generated in an analogous manner to the upper bound in Section \ref{sec: upper-bounding the value function} by first constructing a lower bound for the Mayer cost $g$. To this end, consider the same $N$ levels $\{\gamma_k\}^N_{k=1}$ as Section \ref{sec: upper-bounding the value function} and their associated set of points $\{\Bar{x}_{i,k}\}^{n_k}_{i=s1}$, which each lie on the $\gamma_k$-level set of $g$. Attach to each point $\Bar{x}_{i,k}$ a subgradient $p_{i,k}\in D^-g(\Bar{x}_{i,k})$. Then, the collection of points $\{\Bar{x}_{i,k}\}^{n_k,N}_{i,k=1}$ and their corresponding subgradients $\{p_{i,k}\}^{n_k,N}_{i,k=1}$ may be used to construct a candidate lower bound $\underline{g}:\R^n\rightarrow\R$ defined by
\begin{equation}
    \underline{g}(x)\doteq \max_{k\in\{1,\ldots,N\}}\max_{i\in\{1,\ldots,n_k\}}\underline{g}_{i,k}(x), \quad \underline{g}_{i,k}(x) \doteq \iprod{p_{i,k}}{x} - \iprod{p_{i,k}}{\Bar{x}_{i,k}}+\gamma_k.
    \label{eq: main: lower bound for g}
\end{equation}
The individual functions $\underline{g}_{i,k}:\R^n\rightarrow\R$ are hyperplanes that support the epigraph of $g$ at $\Bar{x}_{i,k}$, which is depicted in Fig. \ref{fig: main: lower bound for terminal cost} for a 1D example.
\begin{figure}[ht!]
    \centering
    \begin{subfigure}[t]{0.39\textwidth}
        \centering
        \includegraphics[width = \columnwidth]{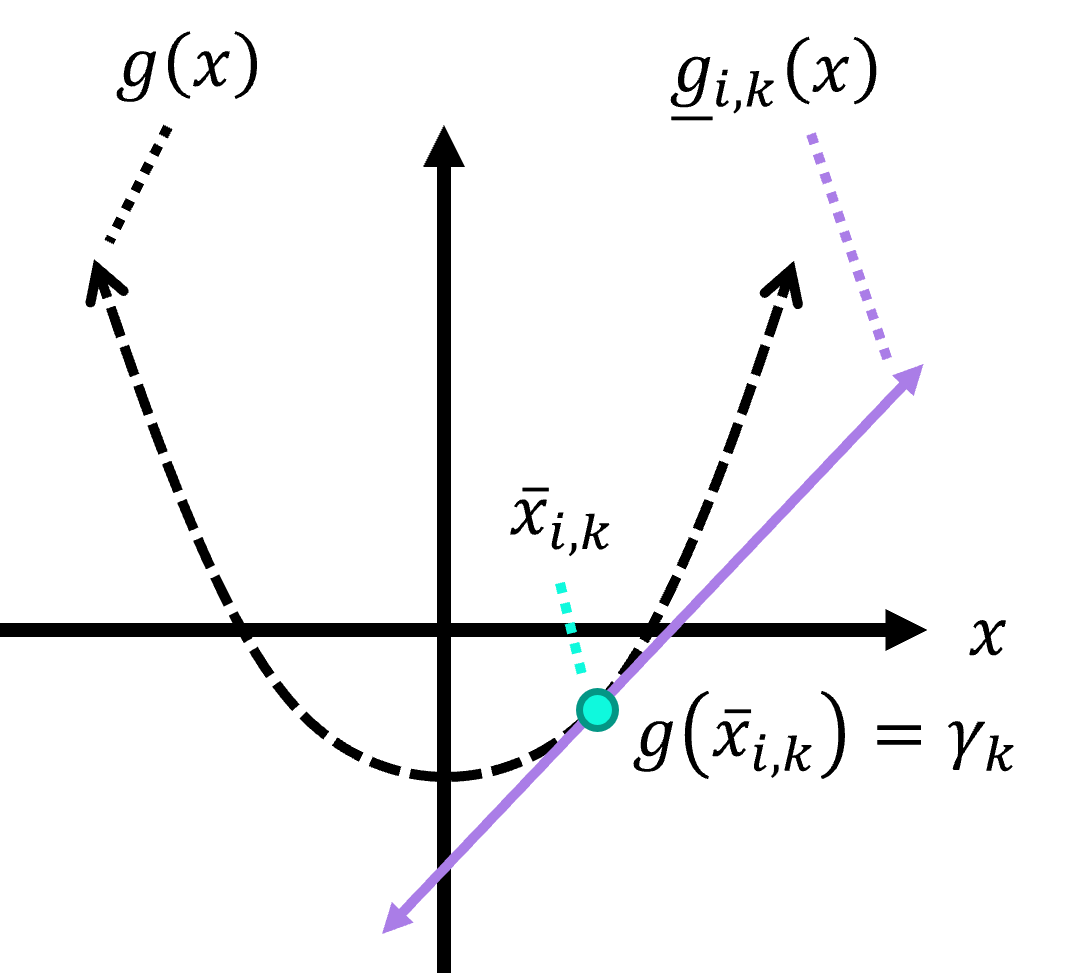}
        \caption{$\underline{g}_{i,k}$}
        \label{subfig: main: single lower bound}
    \end{subfigure}%
    \hspace{0.4cm}
    \begin{subfigure}[t]{0.44\textwidth}
        \centering
        \includegraphics[width = \columnwidth]{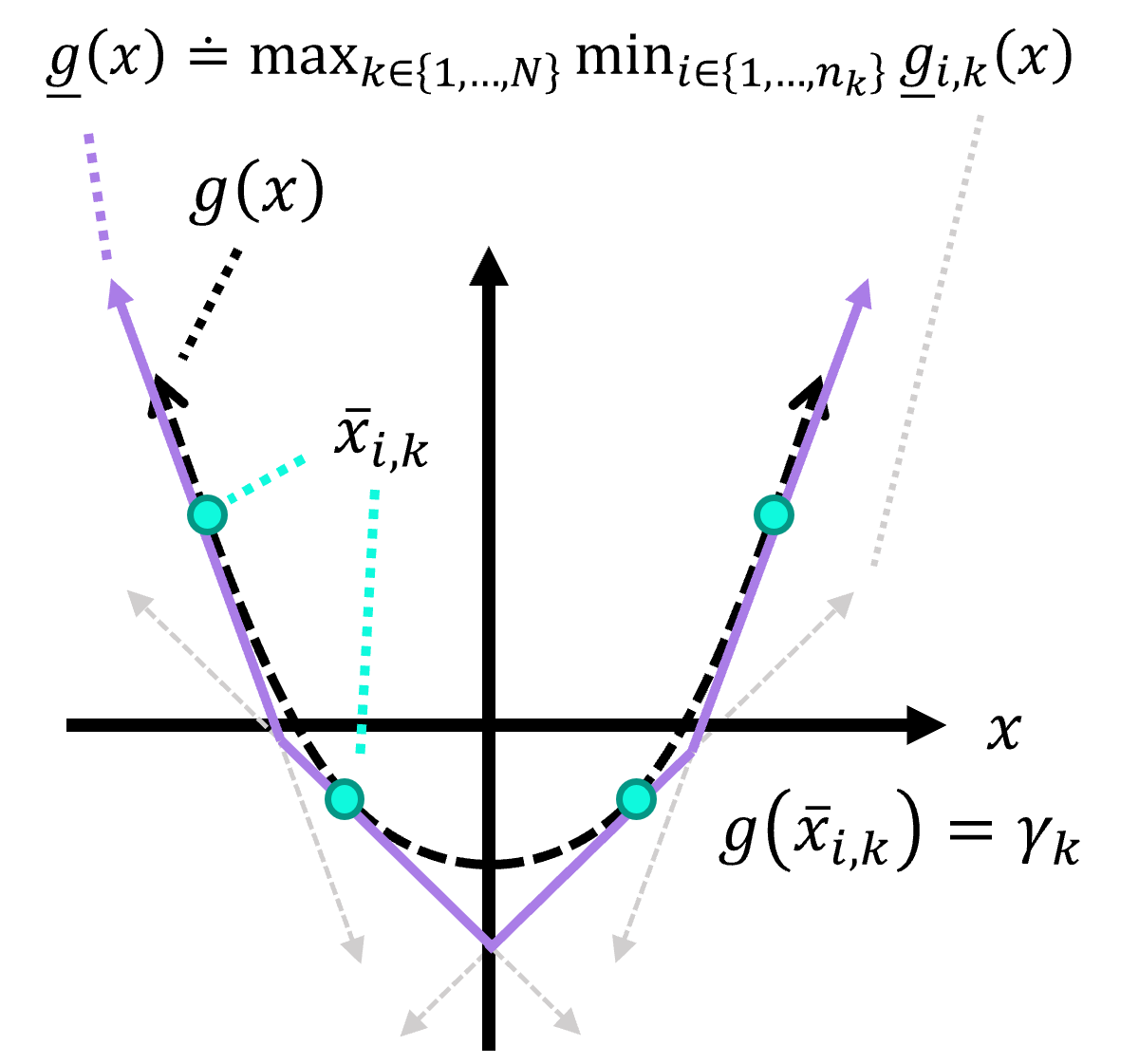}
        \caption{Maximum over $\underline{g}_{i,k}$}
        \label{subfig: main: multiple lower bound}
    \end{subfigure}
    \caption{Illustration of the proposed lower bound $\underline{g}$ in \eqref{eq: main: lower bound for g}, which is the maximum over the individual lower bounds $\underline{g}_{i,k}$. The hyperplanes $\underline{g}_{i,k}$ support the epigraph of $g$ at $\Bar{x}_{i,k}$.}
    \label{fig: main: lower bound for terminal cost}
\end{figure}
\begin{lemma} Let Assumption \ref{assumption: prelim: Mayer cost condition} hold. The function $\underline{g}:\R^n\rightarrow\R$ defined in \eqref{eq: main: lower bound for g} is a lower bound for $g$ for all $x\in\R^n$.
    \label{lemma: main: lower bound for g}
\end{lemma}
\begin{proof} Fix $x\in\R^n$, $k\in\{1,\ldots,N\}$, and $i\in\{1,\ldots,n_k\}$. By \cite[Thm 8.12, p.308]{rockafellar2009variational}, $g(x) \geq g(\Bar{x}_{i,k}) + \iprod{p_{i,k}}{x-\Bar{x}_{i,k}} = \iprod{p_{i,k}}{x} -  \iprod{p_{i,k}}{\Bar{x}_{i,k}} + \gamma_k = \underline{g}_{i,k}(x)$. By \eqref{eq: main: lower bound for g}, the assertion follows by taking the maximum over the indices $k\in\{1,\ldots,N\}$ and $i\in\{1,\ldots,n_k\}$. 
\end{proof}

Next, consider the dynamics
\begin{equation}
        \dot{\lambda}(s) = -A'(s)\lambda(s),\quad \dot{q}(s) = \max_{u\in\U}\min_{d\in\D}\iprod{-\lambda(s)}{B(s)u + E(s)d},\quad \text{a.e. } s\in(t,T),
    \label{eq: main: dynamics for normal vec and offset term}
\end{equation}
and let $\lambda_{i,k}(s) \in \R^n$ and $q_{i,k}(s) \in \R$ denote solutions of \eqref{eq: main: dynamics for normal vec and offset term} at time $s\in[t,T]$ subject to the terminal conditions 
\begin{equation}
        \lambda_{i,k}(T) = p_{i,k} \in D^-g(\Bar{x}_{i,k}),\quad q_{i,k}(T) = -\iprod{p_{i,k}}{\Bar{x}_{i,k}} + \gamma_k,
    \label{eq: main: terminal condition for normal vec and offset term}
\end{equation}
for all $k\in\{1,\ldots,N\}$ and $i\in\{1,\ldots,n_k\}$. It will be shown that the proposed dynamics in \eqref{eq: main: dynamics for normal vec and offset term} ensure that each solution pair $(\lambda_{i,k},q_{i,k})$ define a supporting hyperplane to the epigraph of $x\mapsto v(t,x)$ for all $t\in[0,T]$. The maximisation over all such hyperplanes may be used as a candidate lower bound for $v$ in \eqref{eq: prelim: value function}. In particular, let $\underline{v}:[0,T]\times\R^n$ be defined by 
\begin{equation}
\begin{split}
    &\underline{v}(t,x) \doteq \max_{k\in\{1,\ldots,N\}}\max_{i\in\{1,\ldots,n_k\}}\underline{v}_{i,k}(t,x), \quad \forall (t,x)\in[0,T]\times\R^n,\\
    &\underline{v}_{i,k}(t,x)\doteq \iprod{\lambda_{i,k}(t)}{x} + q_{i,k}(t),
    \end{split}
    \label{eq: main: lower-bounding surface}
\end{equation}
which satisfies the desired terminal condition $\underline{v}(T,x) = \underline{g}(x) \leq g(x)$. The lower bound guarantee is attained by demonstrating that $\underline{v}$ is a viscosity subsolution of \eqref{eq: prelim: HJI}, which is analogous to Lemma \ref{lemma: main: viscosity supersolution via scaled and shifted distance functions}. 

\begin{lemma} Let Assumption \ref{assumption: prelim: Mayer cost condition} hold. The function $\underline{v}$ defined in \eqref{eq: main: lower-bounding surface} is a continuous viscosity subsolution of \eqref{eq: prelim: HJI}.
    \label{lemma: main: viscosity subsolution}
\end{lemma}
\begin{proof}
First note that by Assumption \ref{assumption: prelim: Mayer cost condition} and Remark \ref{remark: main: subdifferential of g}, the set $D^-g(\Bar{x}_{i,k})$ in \eqref{eq: main: terminal condition for normal vec and offset term} is non-empty. Thus, the solution pairs $(\lambda_{i,k},q_{i,k})$, which characterise $\underline{v}$, exist and are well-defined. 

Fix $k\in\{1,\ldots,N\}$, $i\in\{1,\ldots,n_k\}$, and consider $\underline{v}_{i,k}$ defined in \eqref{eq: main: lower-bounding surface}. As the right-hand side of the dynamics in \eqref{eq: main: dynamics for normal vec and offset term} are continuous in time, $\underline{v}_{i,k}\in\sC^1((0,T)\times\R^n;\,\R)$. Next, fix $(t_0,x_0)\in(0,T)\times\R^n$. By \eqref{eq: main: dynamics for normal vec and offset term}, 
    \begin{align}
        &-\tfrac{\partial}{\partial t}\underline{v}_{i,k}(t_0,x_0) + \max_{u\in\U}\min_{d\in\D}\iprod{-\nabla \underline{v}_{i,k}(t_0,x_0)}{A(t_0)x_0 + B(t_0)u+E(t_0)d} \nn \\
        &\qquad =-\big\langle \dot{\lambda}_{i,k}(t_0),x_0\big\rangle -\big\langle\lambda_{i,k}(t_0),A(t_0)x_0\big\rangle  \nn \\
        &\qquad \qquad- \dot{q}_{i,k}(t_0) + \max_{u\in\U}\min_{d\in\D}\iprod{-\lambda_{i,k}(t_0)} {B(t_0)u+E(t_0)d} = 0. \label{eq: main: viscosity subsolution proof 1}
    \end{align}
     Consequently, $\underline{v}_{i,k}$ is a classical solution of \eqref{eq: prelim: HJI}, and hence a viscosity solution, which, by Definition \ref{def: prelim: visc solution}, also implies that it is a viscosity subsolution of \eqref{eq: prelim: HJI}. The assertion holds by \eqref{eq: main: lower-bounding surface} and noting that the maximisation over any finite collection of viscosity subsolutions of \eqref{eq: prelim: HJI} yields another viscosity subsolution (apply induction to \cite[Prop 2.1, p.34]{bardi1997optimal}).
\end{proof}

In addition to being a lower bound for the value function, $\underline{v}$ coincides with the value $v$ along solutions of \eqref{eq: main: PMP trimmed input dynamics} in which the costate trajectories are given by $\lambda_{i,k}$, which parameterise $\underline{v}$. 

\begin{theorem}
    Let Assumption \ref{assumption: prelim: Mayer cost condition} hold. Then,
    \begin{enumerate}[(a)]
        \item \label{item: main: lower bound result 1} $\underline{v}$ as defined in \eqref{eq: main: lower-bounding surface} is a lower bound for the value function $v$ in \eqref{eq: prelim: value function} for all $(t,x)\in[0,T]\times\R^n$.
        \item \label{item: main: lower bound result 2} Additionally, let Assumptions \ref{assumption: prelim: trimmed input condition}-\ref{assumption: prelim: recoverability of player 1 set} hold. Given $k\in\{1,\ldots,N\}$ and $i\in\{1,\ldots,n_k\}$, if $\lambda_{i,k}$, which characterises $\underline{v}$, satisfies the terminal value problem in \eqref{eq: main: PMP trimmed input dynamics} for any pair $(\xi^\star, \omega^\star)$ with $\xi^\star(T) = \Bar{x}_{i,k}$ and $\omega^\star\in\sW[t,T]$, then,
        \begin{equation}
             v\big(t, \xi^\star(t)\big) = \underline{v}\big(t, \xi^\star(t)\big) = \gamma_k.
            \label{eq: main: tight lower bound}
        \end{equation}
    \end{enumerate}
    \label{theorem: main: lower-bounding the value}
\end{theorem}
\begin{proof}
    \underline{Assertion (\ref{item: main: lower bound result 1}).}   By \eqref{eq: main: lower-bounding surface}, the terminal condition in \eqref{eq: main: terminal condition for normal vec and offset term}, and Lemma \ref{lemma: main: lower bound for g}, $\underline{v}(T,x) = \underline{g}(x) \leq g(x)$ for all $x\in\R^n$. By Theorem \ref{theorem: prelim: value function and HJI equation} and Definition \ref{def: prelim: visc solution}, $v$ in \eqref{eq: prelim: value function} is a viscosity supersolution of \eqref{eq: prelim: HJI}. Then, by Lemma \ref{lemma: main: viscosity subsolution} and Theorem \ref{theorem: prelim: comparison result}, assertion (\ref{item: main: lower bound result 1}) holds.

    \underline{Assertion (\ref{item: main: lower bound result 2}).} Let $(\xi^\star,\omega^\star)$ be a pair satisfying \eqref{eq: main: PMP trimmed input dynamics} with the costate solution given by $\lambda_{i,k}$ and $\xi^\star(T) = \Bar{x}_{i,k} \in g^{-1}(\{\gamma_k\})$. Consider the map $s\mapsto \underline{v}_{i,k}\big(s,\xi^\star(s)\big)$ in which $\underline{v}_{i,k}\in\sC^1((0,T)\times\R^n;\,\R)$ is defined in \eqref{eq: main: lower-bounding surface}. Then, for a.e. $s\in(t,T)$,
    \begin{align}
        &\tfrac{d}{ds}\underline{v}_{i,k}\big(s,\xi^\star(s)\big) = \tfrac{\partial}{\partial s}\underline{v}_{i,k}\big(s,\xi^\star(s)\big) + \big\langle \nabla\underline{v}_{i,k}\big(s,\xi^\star(s)\big), \dot{\xi}^\star(s)\big\rangle \nn \\
        & \quad = \big\langle -A'(s)\lambda_{i,k}(s),\xi^\star(s) \big\rangle + \max_{u\in\U}\min_{d\in\D}\iprod{-\lambda_{i,k}(s)}{B(s)u + E(s)d}+\big\langle\lambda_{i,k}(s),A(s)\xi^\star(s) + \omega^\star(s) \big\rangle \nn \\
        &\quad = H(s,\xi^\star(s),\lambda_{i,k}(s)) - \cH(s,\xi^\star(s),\lambda_{i,k}(s)) = 0,
        \label{eq: main: lower-bounding the value proof 1}
    \end{align}
    in which $H$ and $\cH$ are given by \eqref{eq: prelim: hamiltonian} and \eqref{eq: main: equivalence of value functions proof 2}, respectively, and the last equality follows by Lemma \ref{lemma: main: equivalence of value functions}(\ref{item: equivalence of value functions 2}). Thus, $\frac{d}{ds}\underline{v}_{i,k}\big(s,\xi^\star(s)\big) = 0$ for a.e. $s\in(t,T)$. Consequently, by \eqref{eq: main: lower-bounding surface} and \eqref{eq: main: terminal condition for normal vec and offset term},
    \begin{align}
        \underline{v}_{i,k}\big(t,\xi^\star(t)\big) &= \integral{T}{t}{\tfrac{d}{ds}\underline{v}_{i,k}\big(s,\xi^\star(s)\big)}{s} + \underline{v}_{i,k}\big(T,\xi^\star(T)\big) = 0 + \iprod{\lambda_{i,k}(T)}{\xi^\star(T)} + q_{i,k}(T) =\gamma_k, \nn \\
        &\implies   \underline{v}\big(t,\xi^\star(t)\big) \geq \gamma_k.
        \label{eq: main: lower-bounding the value proof 2}
    \end{align}
    Finally, replacing $(\xi^\star_{i,k}, \omega_{i,k})$ with $(\xi^\star, \omega^\star)$ in the arguments of subclaim (b) of Corollary \ref{corollary: main: tight upper bound} yields $v\left(t,\xi^\star(t)\right) = \gamma_k$. Then, by assertion (\ref{item: main: lower bound result 1}) and \eqref{eq: main: lower-bounding the value proof 2}, assertion (\ref{item: main: lower bound result 2}) holds.
\end{proof}

Theorem \ref{theorem: main: lower-bounding the value}(\ref{item: main: lower bound result 2}) and Corollary \ref{corollary: main: tight upper bound} imply that solutions generated via the dynamics in \eqref{eq: main: PMP trimmed input dynamics} and \eqref{eq: main: dynamics for normal vec and offset term} from the set of terminal states $\{\Bar{x}_{i,k}\}^{n_k,N}_{i,k=1}$ may be used to characterise upper and lower bounds $\Bar{v}$ and $\underline{v}$ that both coincide with $v$ along solutions $\xi^\star_{i,k}$ of \eqref{eq: main: PMP trimmed input dynamics}. 
\begin{corollary}
Let Assumptions \ref{assumption: prelim: Mayer cost condition}-\ref{assumption: prelim: recoverability of player 1 set} hold and consider $v$, $\Bar{v}$, and $\underline{v}$ defined in \eqref{eq: prelim: value function}, \eqref{eq: main: upper-bounding surface}, and \eqref{eq: main: lower-bounding surface}, respectively. Then,
\begin{enumerate}[(a)]
    \item \label{item: tight upper and lower bound 1} For any terminal state $\Bar{x}\in\R^n$, \eqref{eq: main: PMP trimmed input dynamics} admits a solution $(\xi^\star, \lambda)$, potentially non-unique, on $s\in[t,T]$ for some $\omega^\star\in\sW[t,T]$; and
    \item \label{item : tight upper and lower bound 2} If every tuple $(\xi^\star_{i,k}, \lambda_{i,k}, \omega_{i,k})$ satisfies \eqref{eq: main: PMP trimmed input dynamics}, then,
    \begin{equation}
        v(t,\xi^\star_{i,k}(t)) = \Bar{v}(t,\xi^\star_{i,k}(t))  = \underline{v}(t,\xi^\star_{i,k}(t)) = \gamma_k,
        \label{eq: main: upper and lower bounds coincide}
    \end{equation}
    for all $k\in\{1,\ldots,N\}$ and $i\in\{1,\ldots,n_k\}$.
\end{enumerate}
    \label{corollary: main: upper and lower bounds coincide}
\end{corollary}
\begin{proof}
    \underline{Assertion (\ref{item: tight upper and lower bound 1}).} Fix $\Bar{x} \in \R^n$. Since $D^-g(\Bar{x})$ is non-empty (Remark \ref{remark: main: subdifferential of g}), a selection for the terminal condition $\lambda(T)\in D^-g(\Bar{x})$ exists, which yields an explicit solution $\lambda(s) = \Phi_{-A'}(s,T)\lambda(T)$, $s\in[t,T]$. Noting that $(s,\omega)\mapsto \langle-\lambda(s),\omega\rangle$ is continuous, by the same arguments below \eqref{eq: main: equivalence of value functions proof 7}, there exists a measurable selection $s\mapsto \omega^\star(s)$ on $[t,T]$ satisfying \eqref{eq: main: PMP trimmed input dynamics}. Then, the state dynamics admit the explicit solution $\xi^\star(s) = \Phi_A(s,T)\Bar{x} + \integral{T}{s}{\Phi_A(s,\tau)\omega^\star(\tau)}{\tau}$.

\underline{Assertion (\ref{item : tight upper and lower bound 2})} holds by Theorem \ref{theorem: main: lower-bounding the value}(\ref{item: main: lower bound result 2}) and Corollary \ref{corollary: main: tight upper bound}. 
\end{proof}

\section{Application: Reachability Analysis}
\label{sec: reachability}
The differential game studied in this work is particularly relevant in the application of reachability analysis. To elaborate, consider a target set $\cX_\gamma$, which can be described by the $\gamma$-sublevel set of $g$, i.e., 
\begin{equation}
    \cX_\gamma \doteq \left\{x\in\mathbb{R}^n \, |\, g(x) \leq \gamma \right\}.
    \label{eq: reachability: target set}
\end{equation}
Suppose a game is played between players I and II, in which player I wishes to reach $\cX_\gamma$ at the terminal time $T$ from some initial point $x$ at time $t$ and player II wishes to avoid $\cX_\gamma$. Player I wins in the case that $x$ is a point such that there exists a control $u\in\sU[t,T]$ that leads the system to $\cX_\gamma$ despite any strategy $\delta\in\Delta[t,T]$ that player II may employ. The set of all such points is given by
\begin{equation}
    \cR(t;\cX_\gamma)  \doteq \left\{x\in\mathbb{R}^n \, | \, \exists\, u\in\sU[t,T] \text{ s.t. } \forall \delta \in \Delta[t,T], \, \xi^{u,\delta[u]}_{t,x}(T) \in \cX_\gamma \right\},
    \label{eq: reachability: maximal BRS}
\end{equation}
which is referred to as the \emph{maximal} backwards reachable set to $\cX_\gamma$ (see \cite[Def 1]{chen2018hamilton}). Analogously, player II wins at points where there exists a strategy $\delta\in\Delta[t,T]$ such that the system avoids $\cX_\gamma$ at the terminal time despite any action $u\in\sU[t,T]$ that player I can take. The \emph{minimal} backwards reachable set (see \cite[Def 2]{chen2018hamilton}) to the complement of the target set, i.e., $\cX^c_\gamma$, characterises these points and is defined by
\begin{equation}
    \cA(t;\cX^c_\gamma)  \doteq \left\{x\in\mathbb{R}^n \, | \, \exists\, \delta\in\Delta[t,T] \text{ s.t. } \forall u \in \sU[t,T], \, \xi^{u,\delta[u]}_{t,x}(T) \in \cX^c_\gamma \right\}.
    \label{eq: reachability: minimal BRS}
\end{equation}
The sets $\cR(t;\cX_\gamma)$ and $ \cA(t;\cX^c_\gamma)$ may be characterised by the $\gamma$-sublevel and superlevel set of the value $v$ in \eqref{eq: prelim: value function}, respectively. The former result is a special case of \cite[Prop 2]{fisac2015reach} and the latter result follows from \cite[Lem 8]{mitchell2005time}, which are recalled in the following. 
\begin{theorem} Consider the value function $v\in\sC([0,T] \times\mathbb{R}^n;\,\mathbb{R})$ defined in \eqref{eq: prelim: value function}, in which $g\in\sC\left(\R^n;\,\R\right)$ characterises the target set $\cX_\gamma$ in \eqref{eq: reachability: target set}. Then, the maximal backwards reachable set to $\cX_\gamma$ (resp. minimal backwards reachable set to $\cX^c_\gamma$) for \eqref{eq: prelim: linear system} is
\begin{equation*}
\begin{split}
       \cR(t;\cX_\gamma) &= \left\{x\in\mathbb{R}^n \, | \, v(t, x) \leq \gamma \right\}, \\
       \cA(t;\cX^c_\gamma) &= \left\{x\in\mathbb{R}^n \, | \, v(t, x) > \gamma \right\},
\end{split}
\end{equation*}
for all $t\in[0,T]$ and $\gamma \in \R$.
    \label{theorem: reachability: maximal and minimal BRS via sublevel and superlevel sets}
\end{theorem}

\begin{remark} The case where $\gamma = 0$ in \eqref{eq: reachability: target set} is considered in \cite{fisac2015reach} and \cite{mitchell2005time}. The generalisation in Theorem \ref{theorem: reachability: maximal and minimal BRS via sublevel and superlevel sets} follows by noting that $\cX_\gamma$ is equivalently the zero sublevel set of $x\mapsto \Tilde{g}(x) \doteq g(x)-\gamma$, which can be used to replace $g$ in \eqref{eq: prelim: value function}, resulting in a translation of $v$.
\end{remark}

Consequently, the proposed upper bound and lower bounds in Sections \ref{sec: upper-bounding the value function} and \ref{sec: lower-bounding the value function} yield inner-approximations of $\cR(t;\cX_\gamma)$ and $ \cA(t;\cX^c_\gamma)$, respectively. These inner-approximations yield conservative estimates for the set of states for which player I and player II can expect to reach $\cX_\gamma$ or, respectively, avoid $\cX_\gamma$ despite the presence of an adversary.
\begin{corollary} 
Let Assumptions \ref{assumption: prelim: Mayer cost condition}-\ref{assumption: prelim: trimmed input condition} hold and consider $\Bar{v}$ and $\underline{v}$ defined in \eqref{eq: main: upper-bounding surface} and \eqref{eq: main: lower-bounding surface}, respectively. Consider also the target set $\cX_\gamma$ in \eqref{eq: reachability: target set}, which characterises $\cR(t;\cX_\gamma)$ and $ \cA(t;\cX^c_\gamma)$ in \eqref{eq: reachability: maximal BRS} and \eqref{eq: reachability: minimal BRS}, respectively. Then,
\begin{align}
    \Tilde{\cR}(t;\cX_\gamma) &\doteq \{x\in\R^n\,|\,\Bar{v}(t,x) \leq \gamma\} \subseteq \cR(t;\cX_\gamma), \label{eq: reachability: upper-bound yields inner-approx for maximal BRS}\\
    \Tilde{\cA}(t;\cX^c_\gamma)&\doteq \{x\in\R^n\,|\,\underline{v}(t,x) > \gamma \} \subseteq \cA(t;\cX^c_\gamma), \label{eq: reachability: lower-bound yields inner-approx for minimal BRS}
\end{align}
for all $t\in[0,T]$ and $\gamma \in \R$.
\label{corollary: reachability: inner approximations of maximal and minimal BRS}
\end{corollary}
\begin{proof} Fix any $t\in[0,T]$ and $\gamma \in \R$. Note that in the case where $\Tilde{\cR}(t;\cX_\gamma)$ is empty \eqref{eq: reachability: upper-bound yields inner-approx for maximal BRS} holds and the same applies for $\Tilde{\cA}(t;\cX^c_\gamma)$. The remaining proof assumes that these sets are non-empty. Let $x\in\Tilde{\cR}(t;\cX_\gamma)$. By Theorem \ref{theorem: main: upper-bounding the value function}, $\gamma \geq \Bar{v}(t,x) \geq v(t,x)$ with $v$ given by \eqref{eq: prelim: value function}. By Theorem \ref{theorem: reachability: maximal and minimal BRS via sublevel and superlevel sets}, $x\in\cR(t;\cX_\gamma)$. As $x\in\Tilde{\cR}(t;\cX_\gamma)$ is arbitrary, $\Tilde{\cR}(t;\cX_\gamma)\subseteq \cR(t;\cX_\gamma)$, which is \eqref{eq: reachability: upper-bound yields inner-approx for maximal BRS}. Next, suppose $x\in\Tilde{\cA}(t;\cX^c_\gamma)$. By Theorem \ref{theorem: main: lower-bounding the value}(\ref{item: main: lower bound result 1}), $\gamma < \underline{v}(t,x) \leq v(t,x)$, which, by Theorem \ref{theorem: reachability: maximal and minimal BRS via sublevel and superlevel sets}, yields $x\in\cA(t;\cX^c_\gamma)$. As $x\in\Tilde{\cA}(t;\cX^c_\gamma)$ is arbitrary, $\Tilde{\cA}(t;\cX^c_\gamma)\subseteq \cA(t;\cX^c_\gamma)$.
\end{proof}

\section{An Algorithmic Procedure}
\label{sec: algorithm}
The upper bound $\Bar{v}$ in \eqref{eq: main: upper-bounding surface} is implicitly characterised via the Mayer cost $g$, the state-transition matrix $\Phi_A(T,t)$, and a collection of solutions $\{\xi^\star_{i,k}\}^{n_k,N}_{i,k=1}$ of \eqref{eq: main: PMP trimmed input dynamics}. By Corollary \ref{corollary: main: tight upper bound}, the upper bound coincides with the value $v$ in \eqref{eq: prelim: value function} along these solutions if $\xi^\star_{i,k}$ are generated via \eqref{eq: main: PMP trimmed input dynamics}, which requires propagation of a collection of costates $\{\lambda_{i,k}\}^{n_k,N}_{i,k=1}$. Together with a collection of scalar solutions $\{q_{i,k}\}^{n_k,N}_{i,k=1}$ of \eqref{eq: main: dynamics for normal vec and offset term}-\eqref{eq: main: terminal condition for normal vec and offset term}, the pairs $\{(\lambda_{i,k}, q_{i,k})\}^{n_k,N}_{i,k=1}$ define a collection of hyperplanes that characterise the proposed lower bound $\underline{v}$ in \eqref{eq: main: lower-bounding surface}. By Corollary \ref{corollary: main: upper and lower bounds coincide}, generating the upper and lower bound using the same costate solutions $\lambda_{i,k}$ results in all three functions $\Bar{v}$, $v$, and $\underline{v}$ coinciding along the solutions $\xi^\star_{i,k}$. 

\begin{algorithm}[!ht]
 \caption{Pre-computation for the Upper and Lower Bounds}
 \label{alg: value function upper and lower bound}
\begin{algorithmic}[1]
  \STATE \textbf{Input:} $\{t, T\}$, $\{\gamma_k\}^N_{k=1}$, $\{n_k\}^N_{k=1}$
  \STATE \textbf{Output:} $\{(\xi^{\star}_{i,k}(t),\lambda_{i,k}(t), q_{i,k}(t))\}^{n_k, N}_{i, k=1}$, $\Phi_A(T,t)$
  \FORALL{$k \in \{1,\ldots, N\}$ and $i \in \{1,\ldots,n_k\}$}
        \STATE \label{alg step: selection of points and subgradients} Select $\Bar{x}_{i,k} \in g^{-1}(\{\gamma_k\})$ and $p_{i,k} \in D^{-}g(\Bar{x}_{i,k})$
  \ENDFOR
  \FORALL{$k\in\{1,\ldots,N\}$ and $i \in \{1,\ldots,n_k\}$}
  \STATE $(\xi^{\star}_{i,k}(t), \lambda_{i,k}(t), q_{i,k}(t)) \leftarrow$ solution of \eqref{eq: main: PMP trimmed input dynamics} and \eqref{eq: main: dynamics for normal vec and offset term} with 
  \begin{equation}
      (\xi^\star(T), \lambda(T), q(T)) = (\Bar{x}_{i,k}, p_{i,k}, -\langle p_{i,k},\Bar{x}_{i,k}\rangle + \gamma_k).
      \label{eq: algorithm: terminal conditions}
  \end{equation}
  \ENDFOR
  \STATE $\Phi_A(T,t) \leftarrow $ solution of 
    \begin{equation}
        \textstyle{\frac{\partial}{\partial s}\Phi_A(T,s)} = -\Phi_A(T,s)A(s),\quad \Phi_A(T,T) = \I_n.\label{eq: algorithm: state-transition matrix}
    \end{equation}
\end{algorithmic}
\end{algorithm}

Based on the results of Section \ref{sec: main results}, a possible numerical scheme to compute the solution tuples $\{\xi^\star_{i,k}(t),\lambda_{i,k}(t),q_{i,k}(t)\}^{n_k,N}_{i,k=1}$ is summarised in Algorithm \ref{alg: value function upper and lower bound}. The algorithm assumes that the system matrices $A(s)$, $B(s)$, and $E(s)$ as well as the trimmed control set $\W(s)$ can be accessed at any time $s\in [t,T]$. An integration scheme such as a Runge-Kutta scheme may be used to integrate the dynamics in \eqref{eq: main: PMP trimmed input dynamics}, \eqref{eq: main: dynamics for normal vec and offset term}, and \eqref{eq: algorithm: state-transition matrix}. The algorithm pre-computes data that may then be used to access the value of the upper and lower bounds online.

\begin{algorithm}[!ht]
 \caption{Accessing the Upper and Lower Bounds}
 \label{alg: accessing value function upper and lower bound}
\begin{algorithmic}[1]
  \STATE \textbf{Input:} $x$, $\{\gamma_k\}^N_{k=1}$, $\{(\xi^{\star}_{i,k}(t),\lambda_{i,k}(t), q_{i,k}(t))\}^{n_k, N}_{i, k=1}$, $\Phi_A(T,t)$
  \STATE \textbf{Output:} $\Bar{v}(t,x)$, $\underline{v}(t,x)$
  \FORALL{$k \in \{1,\ldots, N\}$}
  \STATE \label{step: alg: half space rep} $\left(C_k(t),d_k(t)\right)\leftarrow \texttt{halfSpace}\big(\{\xi^{\star}_{i,k}(t)\}^{n_k}_{i=1}\big)$
  \STATE Solve
  \begin{align}
    \xi^\star_k(t,x) \leftarrow&\argmin_{\xi\in\R^n} \textstyle{\frac{1}{2}}\langle\xi,\, \Phi'_A(T,t)\Phi_A(T,t)\xi\rangle - \langle \Phi'_A(T,t)\Phi_A(T,t)x,\,\xi\rangle,  \label{eq: algorithm: QP}\\
        &\text{subject to }  C_k(t)\xi \leq d_k(t),\nn
\end{align}
\STATE $\Bar{v}_k(t,x) \leftarrow L_g \norm{\Phi_A(T,t)(x-\xi^\star_k(t,x))}+\gamma_k$
  \ENDFOR
    \STATE \label{step: alg: access upper bound} $\Bar{v}(t,x) \leftarrow \min_{k\in\{1,\ldots,N\}}\Bar{v}_k(t,x)$
    \STATE \label{step: alg: access lower bound} $\underline{v}(t,x)\leftarrow \max_{k\in\{1,\ldots,N\}}\max_{i\in\{1,\ldots,n_k\}}\{\iprod{\lambda_{i,k}(t)}{x} + q_{i,k}(t)\}$
\end{algorithmic}
\end{algorithm}

For any point $x\in\R^n$, the upper and lower bounds may be accessed at time $t$ using Algorithm \ref{alg: accessing value function upper and lower bound}. Note that the polytope $\Omega^\star_k(t)$ in \eqref{eq: main: upper-bounding surface kth level set} admits a half-space representation $\{\xi \in\R^{n}\,|\, C_k(t)\xi \leq d_k(t)\}$ in which $C_k(t)\in\R^{n_c \times n}$ and $d_k(t) \in \R^{n_c}$ define a collection of $n_c$ linear inequalities \cite[Chap 4.4.4, pp.80-81]{borrelli2017predictive}. Correspondingly, \texttt{halfSpace} in step \ref{step: alg: half space rep} of Algorithm \ref{alg: accessing value function upper and lower bound} refers to any routine that can convert $\Omega^\star_k(t)$ into its half-space representation. Furthermore, given $x\in\R^n$, a minimising point $\xi \in \Omega^\star(t)$ of $\norm{\Phi_A(T,t)(x-\xi)}^2$ is also a minimising point of $\norm{\Phi_A(T,t)(x-\xi)}$. Thus, by \eqref{eq: main: upper-bounding surface kth level set} and \eqref{eq: main: time-dependent dist function def}, $\Bar{v}_k$ can be evaluated by solving the quadratic program in \eqref{eq: algorithm: QP} and subsequently substituting $\xi^\star_k(t,x)$ for $\xi$ in \eqref{eq: main: time-dependent dist function def}. The upper bound $\Bar{v}$ is then obtained via minimisation over all $\Bar{v}_k$ (step \ref{step: alg: access upper bound}) and the lower bound $\underline{v}$ may be accessed by directly evaluating \eqref{eq: main: lower-bounding surface} (step \ref{step: alg: access lower bound}).

\begin{comment}Furthermore, given $x\in\R^n$, a minimising point $\xi \in \Omega^\star(t)$ of $\norm{\Phi_A(T,t)(x-\xi)}^2$ is also a minimising point of $\norm{\Phi_A(T,t)(x-\xi)}$ since minimisers are preserved under composition with monotonic increasing functions. \end{comment}
\section{Numerical Example}
\label{sec: numerical example}
As an illustrative example for this work, consider the LTV dynamics described by
\begin{equation}
    \Dot{\xi}(s) = \begin{bmatrix}0&1&0\\-k_\xi^2(s)&0&1\\0&0&0\end{bmatrix}\xi(s) + \begin{bmatrix}0&0 \\ 1&0\\0&1\end{bmatrix}u(s) + \begin{bmatrix}
        0\\k_d(s)\\0
    \end{bmatrix}d(s), \quad \text{a.e. } s\in(0,1.5), 
   \label{eq: numerical ex: LTV dynamics}
\end{equation}
in which $k_\xi(s) = \sqrt{4 + 2 \cos{(2s)}}$, $k_d(s)=0.5+0.5\sin{(\frac{\pi}{2}s)}$, $\xi(s)=[\xi_1(s), \xi_2(s), \xi_3(s)]'\in\mathbb{R}^3$, $u(s)\in\U\subset\R^2$, and $d(s)\in\D\subset\R$, all at time $s$. The control set for each player is given by 
\begin{equation}
\U = [-1,\,1]\times [-1,\,1], \quad \D = [-1,\,1].
    \label{eq: numerical ex: input sets}
\end{equation}
It can be verified that $\W(s)$ as defined in \eqref{eq: prelim: trimmed input set def} is given by $\W(s) = \{0\}\times [
    -1+k_d(s),\, 1-k_d(s)]\times [-1,\,1]$ for all $s\in[0,1.5]$ for \eqref{eq: numerical ex: LTV dynamics}-\eqref{eq: numerical ex: input sets}. In particular, Assumptions \ref{assumption: prelim: trimmed input condition}-\ref{assumption: prelim: recoverability of player 1 set} are satisfied. Additionally, attach to \eqref{eq: numerical ex: LTV dynamics} the Mayer cost
\begin{equation}
     g(x) = \lVert x\rVert, \quad \forall x\in\R^n,\label{eq: numerical ex: terminal cost}
\end{equation}
which yields a Lipschitz constant for $g$ of $L_g = 1$.

To implement Algorithm \ref{alg: value function upper and lower bound}, consider a selection of $N = 5$ level sets corresponding to the levels $\{\gamma_k\}^N_{k=1} = \{0, 0.3,0.6,0.9,1.2\}$ and $\{n_k\}^N_{k=1} = \{85, 84, 84, 84, 84\}$ solutions for each level set. The points $\Bar{x}_{i,k}$ and their associated subgradients $p_{i,k}$ may be selected in step \ref{alg step: selection of points and subgradients} of Algorithm \ref{alg: value function upper and lower bound} using the fact that
\begin{equation}
    g^{-1}(\{\gamma\}) = \{\gamma e\in\R^n\,|\,\norm{e} = 1\}, \quad D^-g(x) = \begin{cases}
        \{e\in\R^n\,|\,\norm{e}\leq1\}, \quad &x = 0,\\ \left\{\frac{x}{\norm{x}}\right\},\quad &\text{otherwise},
    \end{cases} 
\end{equation}
for $g$ in \eqref{eq: numerical ex: terminal cost} and $\gamma \geq 0$. Using a Runge-Kutta scheme with a step size of 8.3ms to integrate solutions of \eqref{eq: main: PMP trimmed input dynamics}, \eqref{eq: main: dynamics for normal vec and offset term}, \eqref{eq: algorithm: state-transition matrix}, the resulting upper and lower bounds $\Bar{v}$ and $\underline{v}$ at $t = 0$ are displayed in Fig. \ref{fig: numerical ex: comparison of bounds and exact value function}. The intersection of the bounds with the hyperplanes $x_2 = x_3 = 0$, $x_1 = x_3 = 0$, and $x_1 = x_2 = 0$, is displayed in Figs. \ref{subfig: numerical ex: x1 slice}, \ref{subfig: numerical ex: x2 slice}, and \ref{subfig: numerical ex: x3 slice}, respectively. 

\begin{figure}[!ht]
    \centering
    % First row of two subfigures
    \begin{minipage}{0.49\textwidth}
        \centering
        \includegraphics[width=\textwidth]{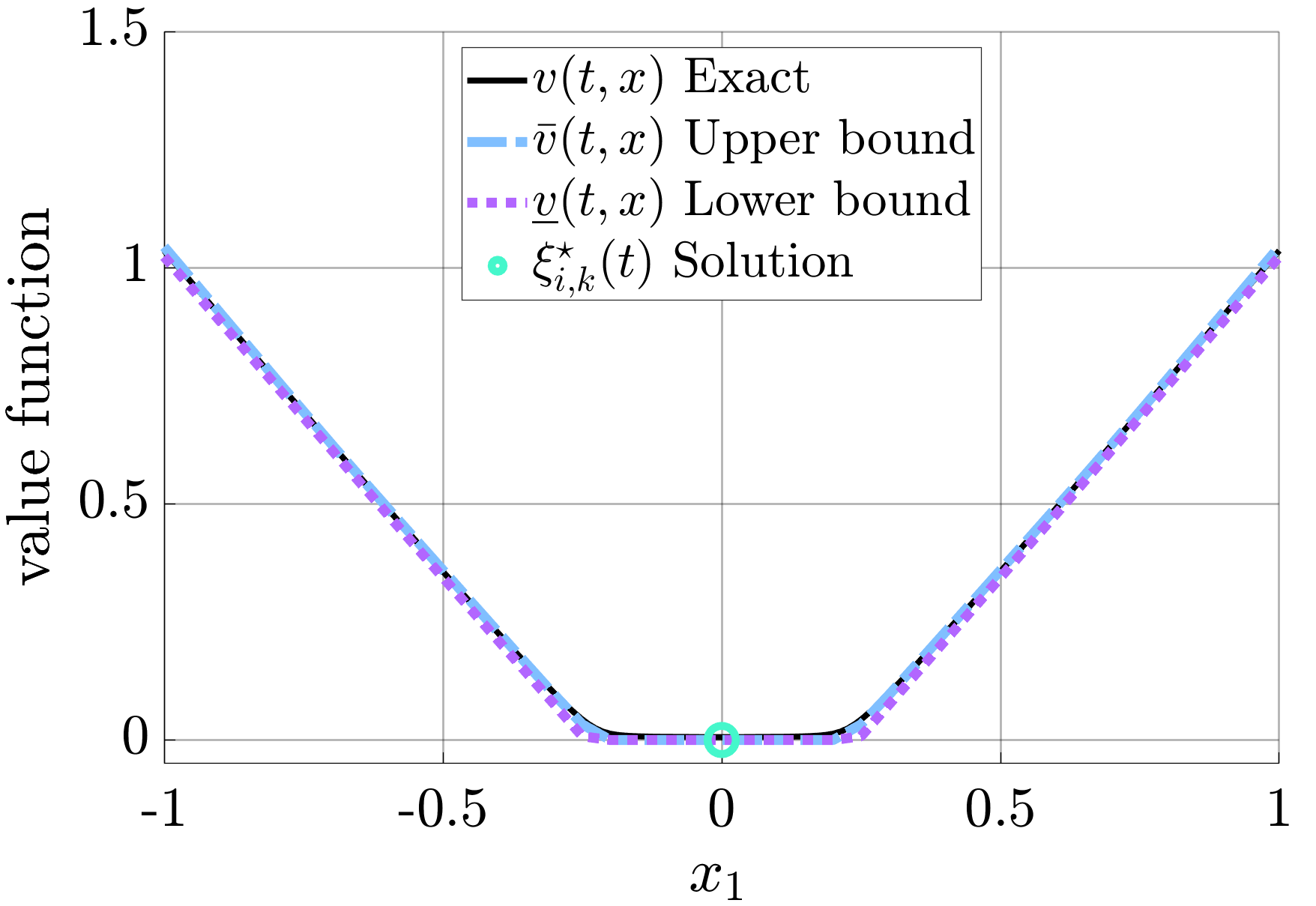} 
        \subcaption{Slice along $x_2 = x_3 = 0$}
        \label{subfig: numerical ex: x1 slice}
    \end{minipage}
    \hfill
    \begin{minipage}{0.49\textwidth}
        \centering
        \includegraphics[width=\textwidth]{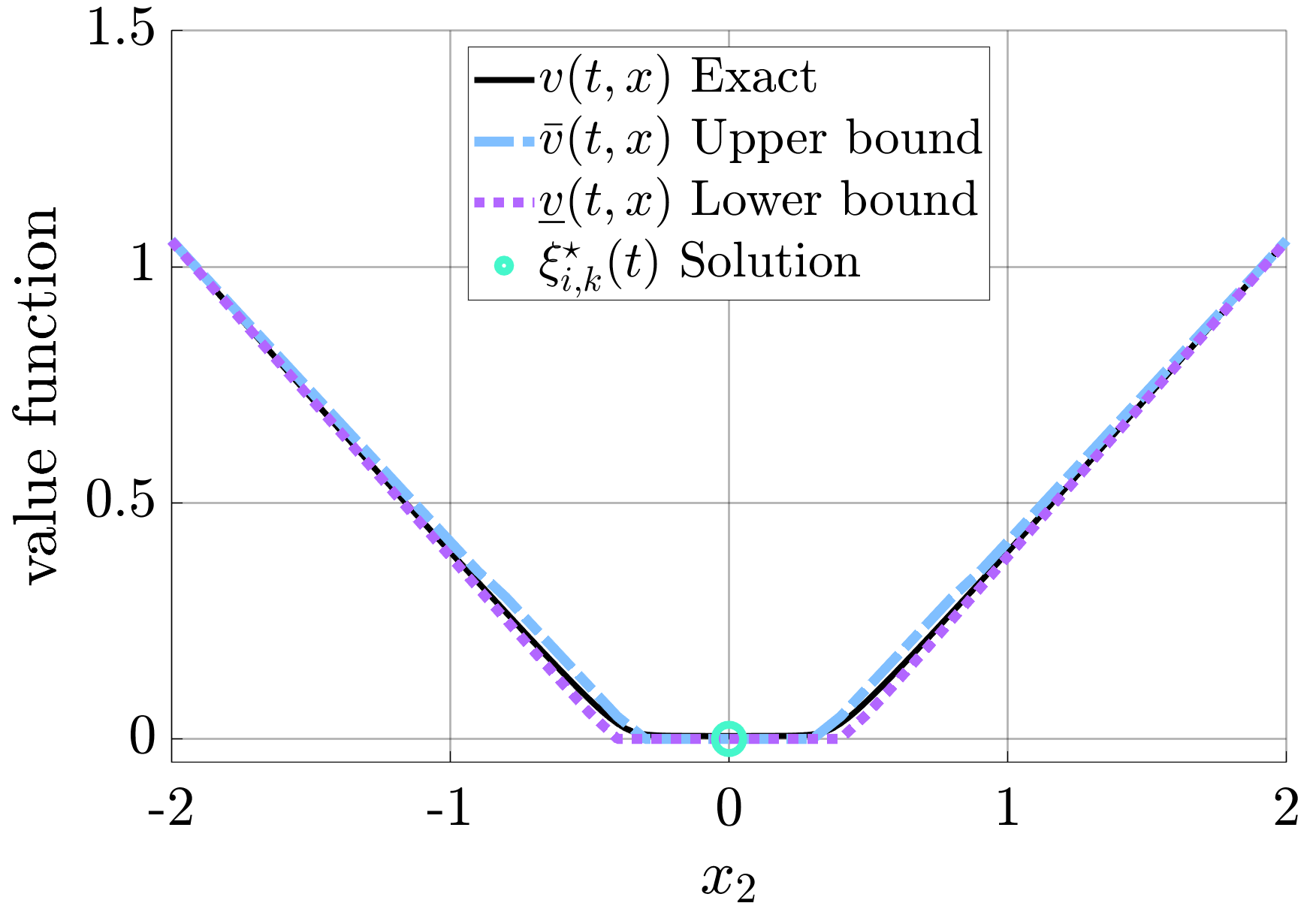} 
        \subcaption{Slice along $x_1 = x_3 = 0$}
                \label{subfig: numerical ex: x2 slice}
    \end{minipage}
    % Second row with one centered subfigure
    \begin{minipage}{0.62\textwidth}
        \centering
        \includegraphics[width=\textwidth]{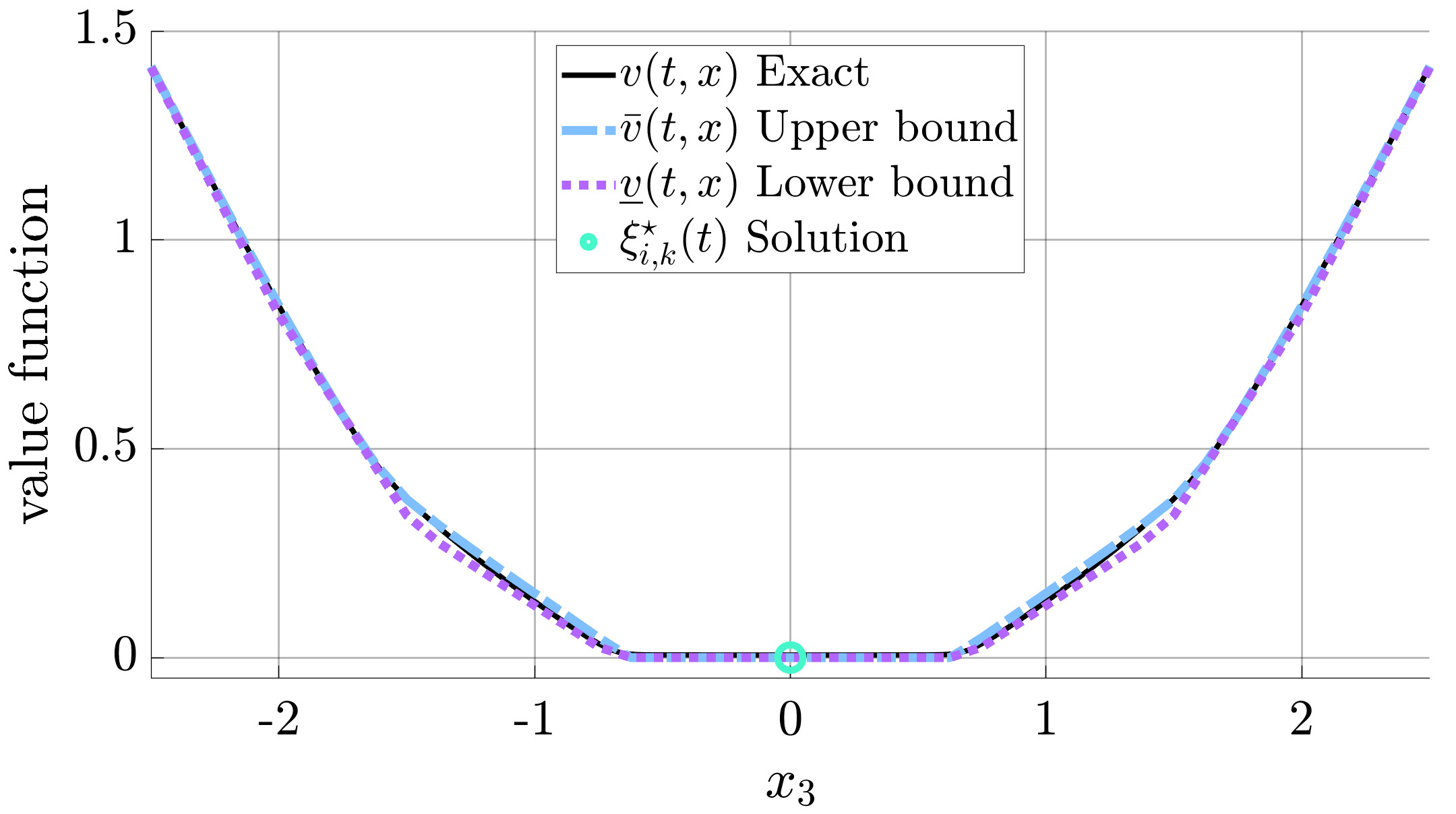} 
        \subcaption{Slice along $x_1 = x_2 = 0$}
                \label{subfig: numerical ex: x3 slice}
    \end{minipage}
    \caption{Comparison of the value function $v(t,x)$ and its upper and lower bounds $\Bar{v}(t,x)$ and $\underline{v}(t,x)$ for \eqref{eq: numerical ex: LTV dynamics}-\eqref{eq: numerical ex: input sets} with the cost in \eqref{eq: numerical ex: terminal cost} at time $t=0$.}
    \label{fig: numerical ex: comparison of bounds and exact value function}
\end{figure}

The visualisation of the bounds has been generated by evaluating $\Bar{v}$ and $\underline{v}$ along a collection of grid points using Algorithm \ref{alg: accessing value function upper and lower bound}. The routines \texttt{vert2lcon} and \texttt{quadprog} in \texttt{MATLAB} have been used to convert $\Omega^\star_k(t)$ to its half-space representation in step \ref{step: alg: half space rep} of Algorithm \ref{alg: accessing value function upper and lower bound} and to solve the quadratic program in \eqref{eq: algorithm: QP}. Note that in a practical application, $\Bar{v}$ and $\underline{v}$ would only need to be evaluated online at a particular point of interest, e.g., the current state of the system. For comparison, the `exact' value function $v$ has been numerically computed on a grid of $151\times151\times151$ points using the finite-difference scheme provided by \cite{mitchelltoolbox}. This value function is included in Fig. \ref{fig: numerical ex: comparison of bounds and exact value function} and the upper and lower bounds, $\Bar{v}$ and $\underline{v}$, have been evaluated along the same grid points.

By Theorems \ref{theorem: main: upper-bounding the value function} and \ref{theorem: main: lower-bounding the value}, the value function $v$ is expected to lie between $\Bar{v}$ and $\underline{v}$, which is observed in Fig. \ref{fig: numerical ex: comparison of bounds and exact value function}. Algorithm \ref{alg: value function upper and lower bound} has been implemented in Fig. \ref{fig: numerical ex: comparison of bounds and exact value function} such that the set $\{\xi^\star_{i,k}(t)\}^{n_k,N}_{i,k=1}$ contains the solution $\xi^\star_{i,k}(t)= 0 \in g^{-1}(\{0\})$ for all $t\in[0,T]$. Thus, by Corollary \ref{corollary: main: upper and lower bounds coincide}, the bounds and the exact value function are expected to coincide at the origin, which is also observed in Fig. \ref{fig: numerical ex: comparison of bounds and exact value function}. These functions must also coincide along all other solutions in the set $\{\xi^\star_{i,k}(\cdot)\}^{n_k,N}_{i,k=1}$, however, these can not be displayed on the same slice simultaneously. 

\pagebreak
\subsection{Performance Evaluation}
Using a fourteen-core Intel{\textregistered} Core\textsuperscript{\texttrademark} i7-13700H CPU, computation times for Algorithm \ref{alg: value function upper and lower bound} and the grid-based approach as well as their memory requirements are displayed in Table \ref{tab: numerical ex: computation times and storage reqs}. The online access time of the upper bound $\Bar{v}$ can be reduced by converting each polytope $\Omega^\star_k(t)$ in \eqref{eq: main: upper-bounding surface kth level set} into its half-space representation offline before solving \eqref{eq: algorithm: QP}. The computation time and memory requirement with this conversion is also included in Table \ref{tab: numerical ex: computation times and storage reqs}. The upper and lower bounds $\Bar{v}$ and $\underline{v}$ may be used as estimates for the value function $v$, with their difference providing a rigorous error bound (to the tolerance of the ODE scheme used and the numerical solver for \eqref{eq: algorithm: QP}). Relative to the maximum value of $v$, the error $\Bar{v} - \underline{v}$ does not exceed $5\%$ over the grid in Fig. \ref{fig: numerical ex: comparison of bounds and exact value function}. These errors can be further reduced by increasing the number of levels $N$ and/or the number of solutions $n_k$ along each level set. 

\renewcommand{\arraystretch}{1.2}
\begin{table}[!ht]
\caption{Comparison of computation times and memory requirements for $v$, $\Bar{v}$, and $\underline{v}$.}
\vspace{-0.2cm}
\begin{center}
\begin{tabular}{|c|p{4cm}|p{4.7cm}|} \hline
\cellcolor{lightgray} \textbf{Numerical method} & \cellcolor{lightgray} \textbf{Computation time} [s] (rel. to \cite{mitchelltoolbox}) & \cellcolor{lightgray} \textbf{Memory requirement} [B] (rel. to \cite{mitchelltoolbox}) \\ \hline \hline
Grid-based approach \cite{mitchelltoolbox} & 1707 (100\%) & 
$7.547\times10^9$ (100\%) \\ \hline
Algorithm \ref{alg: value function upper and lower bound} & $1.203$ (0.07047\%) & $6.948\times10^6$ (0.09207\%) \\ \hline
Algorithm \ref{alg: value function upper and lower bound} (half-space rep.) & $2.219$ (0.1300\%) & $1.312\times10^7$ (0.1738\%)\\ \hline
\end{tabular}
\vspace{-0.4cm}
\end{center}
\label{tab: numerical ex: computation times and storage reqs}
\end{table}
\renewcommand{\arraystretch}{1}

Algorithm \ref{alg: value function upper and lower bound} offers marked improvements in both offline computation and memory requirement compared to grid-based approaches whilst having modest error. This efficiency arises due to the sparse parametrisation of the proposed upper and lower bounds. In particular, $\Bar{v}$ in Fig. \ref{fig: numerical ex: comparison of bounds and exact value function} is characterised by a set of $421$ solutions $\{\xi^\star_{i,k}(\cdot)\}^{n_k,N}_{i,k=1}$ of \eqref{eq: main: LTV system with trimmed input} and a state-transition matrix, and $\underline{v}$ is characterised by a set of $421$ solution pairs $\{\lambda_{i,k}(\cdot), q_{i,k}(\cdot)\}^{n_k,N}_{i,k=1}$ of \eqref{eq: main: dynamics for normal vec and offset term}, which are in contrast to the $151\times151\times151$ grid points characterising $v$. 

The bulk of the computation time and memory requirement for the bounds $\Bar{v}$ and $\underline{v}$ scales in the number of solutions $\sum^N_{k=1}n_k$, which is not explicitly tied to the number of system states. Thus, further savings are expected for higher dimensional systems. It is worth noting that the computational expense required to access $\Bar{v}$ can become significant with increases to the number of solutions as the number of linear inequalities required to represent $\Omega^\star_k(t)$ in \eqref{eq: main: upper-bounding surface kth level set} will also increase \cite{klee1964number}. In Fig. \ref{fig: numerical ex: comparison of bounds and exact value function}, the maximum and average time to access $\Bar{v}$ at each point was $31.25$ms and $10.16$ms, respectively. In a high-dimensional application, techniques for reducing the number of vertices in $\Omega^\star_k(t)$ such as \cite{reisner2001dropping} may be used to alleviate the online computational expense. Although reducing the parametrisation of the bounds may lead to added conservatism, the upper and lower bound guarantees in Theorems \ref{theorem: main: upper-bounding the value function} and \ref{theorem: main: lower-bounding the value} will hold regardless. Thus, the proposed scheme in this work allows for a flexible allocation of computational resources whilst retaining rigorous guarantees in an applied setting.

\section{Conclusion}
\label{sec: conclusion}
In this paper, upper and lower bounds for the value function for a class of constrained linear-time varying games were characterised and an algorithmic procedure for computing these bounds is described. Classical comparison results for viscosity solutions of Hamilton-Jacobi equations were utilised to obtain the upper and lower bound guarantees. The proposed upper bound is characterised by a collection of solutions of a single-player dynamical system and the proposed lower bound is characterised by a collection of hyperplanes. Under appropriate conditions, both the upper and lower bounds can be made to coincide with the value function along the aforementioned collection of solutions. In the application of reachability analysis, these bounds characterise inner-approximations for the set of states that can either reach or avoid a target set despite the presence of an adversary. The upper and lower bounds may be evaluated online by solving either a quadratic program or, respectively, taking the maximum over a collection of hyperplanes. In the numerical example considered in this work, only a sparse collection of solutions and hyperplanes is needed to generate modest error bounds. This in turn produced marked improvements in both the computational expense and memory requirement compared to grid-based approaches. 

\appendix
\section{Supplementary Results}
Definitions for open and closed set-valued maps are recalled in the following from \cite{hogan1973point}, in which $\A$ and $\B$ are arbitrary metric spaces. These definitions are also consistent with, respectively, inner semicontinuity and outer semicontinuity of set-valued maps in \cite[Def 5.4, p.152]{rockafellar2009variational}.
\begin{definition}
    A set-valued map $\Omega:\A\rightarrow\cP\left(\B\right)$\footnote{The set of all subsets of $\B$, i.e., the power set of $\B$, is denoted by $\cP\left(\B\right)$.} is open at a point $\Bar{a}\in\A$ if $\{a_j\}^\infty_{j=1}\subseteq \A$, $a_j\rightarrow \Bar{a}$, and $\Bar{b}\in\Omega(\Bar{a})$ implies the existence of $M\in\N$ and a sequence $\{b_j\}^\infty_{j=1}\subseteq \B$ such that $b_j\in\Omega(a_j)$ for all $j\geq M$ and $b_j\rightarrow \Bar{b}$.
    \label{def: open set-valued maps}
\end{definition}
\begin{definition}
    A set-valued map $\Omega:\A\rightarrow\cP\left(\B\right)$ is closed at a point $\Bar{a}\in\A$ if $\{a_j\}^\infty_{j=1}\subseteq \A$, $a_j\rightarrow\Bar{a}$, $b_j\in\Omega(a_j)$, and $b_j\rightarrow\Bar{b}$, implies $\Bar{b}\in\Omega(\Bar{a})$.
    \label{def: closed set-valued maps}
\end{definition}
\begin{definition}
A set-valued map $\Omega:\A\rightarrow\cP\left(\B\right)$ is continuous at a point $\Bar{a}\in\A$ if it is both open and closed at $\Bar{a}$.
        \label{def: continuous set-valued map}
\end{definition}
The following is recalled from \cite[Thm 7]{hogan1973point}.
\begin{theorem}
If $\Omega:\A\rightarrow\cP\left(\B\right)$ is continuous at $\Bar{a}\in\A$, uniformly compact in a neighbourhood of $\Bar{a}$, and $h:\A\times\B \rightarrow[-\infty,+\infty]$ is continuous on $\Bar{a}\times \Omega(\Bar{a})$, then $\hat{h}:\A\rightarrow [-\infty,+\infty]$, defined by
\begin{equation}
    \hat{h}(a) \doteq \sup_{b\in\Omega(a)}h(a,b), \quad \forall a\in\A,
    \label{eq: supremal value function}
\end{equation}
is continuous at $\Bar{a}$.
    \label{theorem: continuity of supremal value functions}
\end{theorem}

\begin{lemma} Let Assumption \ref{assumption: prelim: trimmed input condition} hold. Given any $k\in\{1,\ldots,N\}$, the function $\Bar{v}_k$ defined in \eqref{eq: main: upper-bounding surface kth level set} is continuous on $[0,T]\times\R^n$.
    \label{lemma: continuity of upper-bounding surface kth level set}
\end{lemma}
\begin{proof}
Let the set-valued map $\hat{\Omega}^\star_k:[0,T]\times\R^n\rightarrow\cP(\R^n)$ be defined by
\begin{equation*}
\hat{\Omega}^\star_k(t,x) \doteq \Omega^\star_k(t) = \mathrm{conv}\left(\{\xi^\star_{i,k}(t)\}^{n_k}_{i=1}\right) , \quad \forall (t,x)\in[0,T]\times\R^n,
\end{equation*}
in which the set $\{\xi^\star_{i,k}(t)\}^{n_k}_{i=1}$ characterises $\Bar{v}_k$ in \eqref{eq: main: upper-bounding surface kth level set}. In particular, $\hat{\Omega}^\star_k$ is constant in $x$. Fix $(\Bar{t},\Bar{x}) \in [0,T]\times\R^n$. 
    
    \emph{Subclaim (a): The set-valued map $\hat{\Omega}^\star_k:[0,T]\times\R^n\rightarrow \cP\left(\R^n\right)$ is uniformly compact in a neighbourhood of $(\Bar{t},\Bar{x})$.} Fix $r>0$ and let $\cY \doteq \B_r((\Bar{t},\Bar{x}))\cap ([0,T]\times\R^n)$. By continuity of solutions of \eqref{eq: main: LTV system with trimmed input}, for any $i\in\{1,\ldots,n_k\}$, $\xi^\star_{i,k}(t)$ is bounded uniformly in $t\in[0,T]$ and consequently bounded uniformly in $(t,x)\in \cY$, thus $\bigcup_{(t,x)\in \cY}\hat{\Omega}^\star_k(t,x) \subset \R^n$ is bounded. Consequently, $\mathrm{cl}\big(\bigcup_{(t,x)\in \cY}\hat{\Omega}^\star_k(t,x)\big)$ is compact, which is subclaim (a). 

    \emph{Subclaim (b): $\hat{\Omega}^\star_k$ is open at $(\Bar{t},\Bar{x})$ (see Definition \ref{def: open set-valued maps}).} Fix any sequence $\{(t_j,x_j)\}^\infty_{j=1}\subseteq [0,T]\times\R^n$ with $(t_j,x_j)\rightarrow (\Bar{t},\Bar{x})$ and fix $\Bar{\xi} \in \hat{\Omega}^\star_k(\Bar{t},\Bar{x}) = \Omega^\star_k(\Bar{t})$. By \eqref{eq: main: upper-bounding surface kth level set}, there exists a set of scalars $\{\kappa_i\}^{n_k}_{i=1}\subset [0,1]$ with $\sum^{n_k}_{i=1}\kappa_i = 1$ such that $\Bar{\xi} = \sum^{n_k}_{i=1}\kappa_i \xi^\star_{i,k}(\Bar{t})$. Consider the sequence $\xi_j \doteq \sum^{n_k}_{i=1}\kappa_i \xi^\star_{i,k}(t_j)$ for all $j\in\N$. By continuity of the solutions $\xi^\star_{i,k}$, $\lim_{j\rightarrow \infty}\xi_j = \Bar{\xi}$. Moreover, as $\xi_j$ is a convex combination of $\{\xi^\star_{i,k}(t_j)\}^{n_k}_{i=1}$, $\xi_j \in \Omega^\star_k(t_j) = \hat{\Omega}^\star_k(t_j,x_j)$, which implies $\hat{\Omega}^\star_k$ is open at $(\Bar{t},\Bar{x})$, i.e., subclaim (b) holds.

    \emph{Subclaim (c): $\hat{\Omega}^\star_k$ is closed at $(\Bar{t},\Bar{x})$ (see Definition \ref{def: closed set-valued maps}).} Fix any sequence $\{(t_j,x_j)\}^\infty_{j=1}\subseteq [0,T]\times\R^n$ such that $(t_j,x_j)\rightarrow (\Bar{t},\Bar{x})$ and fix any sequence $\{\xi_j\}^\infty_{j=1}\subset\R^n$ that converges to $\Bar{\xi}\in\R^n$ with $\xi_j \in \hat{\Omega}^\star_k(t_j,x_j) = \Omega^\star_k(t_j)$ for all $j\in\N$. By \eqref{eq: main: upper-bounding surface kth level set}, there exists a set of scalars $\{\kappa_{i,j}\}^{n_k}_{i=1}\subset [0,1]$ with $\sum^{n_k}_{i=1}\kappa_{i,j} = 1$ such that $\xi_j = \sum^{n_k}_{i=1}\kappa_{i,j}\xi^\star_{i,k}(t_j)$. Let $\cK \doteq \{K \in [0,1]^{n_k}\,|\, \mathbf{1}_{1\times n_k}K = 1\}$ denote the set of coefficients for a convex combination of $n_k$-elements. As $\cK$ is compact and $K_j \doteq \begin{bmatrix}\kappa_{1,j} & \cdots &\kappa_{n_k,j}\end{bmatrix}' \subset \cK$, by Bolzano-Weierstrass, there exists a subsequence of $\{K_j\}^\infty_{j=1}$ that converges to $\Bar{K} \in \cK$. Without loss of generality, let $\{K_j\}^\infty_{j=1}$ be such a sequence. Then,
    \begin{align*}
        \lim_{j\rightarrow \infty} \xi_j 
        &= \lim_{j\rightarrow \infty} \left(\begin{bmatrix}
            \xi^\star_{1,k}(t_j) & \cdots & \xi^\star_{n_k,k}(t_j)
        \end{bmatrix}K_j\right), \\
        &= \begin{bmatrix}
            \xi^\star_{1,k}(\Bar{t}) & \cdots & \xi^\star_{n_k,k}(\Bar{t})
        \end{bmatrix}\Bar{K} \in \text{conv}(\{\xi^\star_{i,k}(\Bar{t})\}^{n_k}_{i=1}).
    \end{align*}
    Thus, $\Bar{\xi} \in \Omega^\star_k(\Bar{t}) = \hat{\Omega}^\star_k(\Bar{t},\Bar{x})$, which implies $\hat{\Omega}^\star_k$ is closed at $(\Bar{t},\Bar{x})$, i.e., subclaim (c) holds. Moreover, by Definition \ref{def: continuous set-valued map} and subclaims (b) and (c), $\hat{\Omega}^\star_k$ is continuous at $(\Bar{t},\Bar{x})$. 
    
    Finally, consider the map $(t,x)\mapsto -d_A(t,x;\Omega^\star_k(t))$, which, by \eqref{eq: main: time-dependent dist function def}, satisfies
    \begin{equation}
        -d_A(t,x;\Omega^\star_k(t)) = -\inf_{\xi\in\Omega^\star_k(t)} \lVert \Phi_A(T,t)(x-\xi)\rVert = \sup_{\xi\in\hat{\Omega}^\star_k(t,x)}-\lVert \Phi_A(T,t)(x-\xi)\rVert. 
    \end{equation}
   By subclaims (a), (b), (c), and using $(t,x)$ in place of $a$ and $\xi$ in place of $b$ in Theorem \ref{theorem: continuity of supremal value functions}, the map $(t,x)\mapsto -d_A(t,x;\Omega^\star_k(t))$ is continuous at $(\Bar{t},\Bar{x})$. Then, by \eqref{eq: main: upper-bounding surface kth level set} and noting that $(\bar{t},\Bar{x})\in[0,T]\times\R^n$ is arbitrary, $\Bar{v}_k$ is continuous on $[0,T]\times\R^n$.
\end{proof}

\begin{lemma} The set $\W(t)$ in \eqref{eq: prelim: trimmed input set def} is convex for all $t\in[0,T]$.
    \label{lemma: convexity of trimmed input set}
\end{lemma}
\begin{proof}
The proof proceeds in the case where $\W(t)$ is non-empty, as empty sets are convex by definition. Fix $t\in[0,T]$, $\kappa \in [0,1]$, $d\in\D$ and let $\omega_1,\omega_2\in\W(t) $. By \eqref{eq: prelim: trimmed input set def}, there exists $u_1,u_2\in\U$ such that $\omega_1-E(t)d = B(t)u_1$ and $\omega_2-E(t)d = B(t)u_2$. Consequently,
\begin{align}
    \kappa \omega_1 + (1-\kappa)\omega_2 - E(t)d &= \kappa (\omega_1-E(t)d) + (1-\kappa)(\omega_2 - E(t)d) \nn \\
    &= B(t)\left(\kappa u_1 + (1-\kappa)u_2\right) = B(t) \hat{u}, \nn
\end{align}
in which $\hat{u} \doteq \kappa u_1 + (1-\kappa)u_2 \in\U$, which follows by convexity of $\U$. As $d\in\D$ is arbitrary, $\{\kappa \omega_1 + (1-\kappa)\omega_2\}\oplus \left(-E(t)\D\right)\subseteq B(t)\U$, thus $\kappa \omega_1 + (1-\kappa)\omega_2 \in \W(t)$. The assertion holds by noting that $t\in[0,T]$ and $\kappa \in [0,1]$ are arbitrary. 
\end{proof}

\begin{lemma} Let Assumption \ref{assumption: prelim: trimmed input condition} hold. Then, the set-valued map $\W:[0,T]\rightarrow \cP\left(\R^{n}\right)$ defined by \eqref{eq: prelim: trimmed input set def} is continuous for all $t\in[0,T]$.
    \label{lemma: continuity of trimmed control set}
\end{lemma}
\begin{proof}
Fix $\Bar{t}\in[0,T]$. 

\emph{Subclaim (a): $\W$ is open at $\Bar{t}$ (see Definition \ref{def: open set-valued maps}).} Fix any sequence $\{t_j\}^\infty_{j=1}\subseteq [0,T]$ with $t_j \rightarrow \Bar{t}$ and let $\Bar{\omega} \in \W(\Bar{t})$. Note that $\Bar{\omega}$ exists by Assumption \ref{assumption: prelim: trimmed input condition}. Next, fix $d\in\D$ and note that by \eqref{eq: prelim: trimmed input set def}, there exists $\hat{u}\in\U$ such that $\Bar{\omega} -E(\Bar{t})d = B(\Bar{t})\hat{u}$. Consequently, for all $j\in\N$,
\begin{equation}
\Bar{\omega} - E(t_j)d + (E(\Bar{t})-E(t_j))d + (B(t_j) - B(\Bar{t}))\hat{u} = B(t_j)\hat{u}. \label{eq: continuity of trimmed control set proof 1}
\end{equation}
By continuity of $B$ and $E$, given $\epsilon > 0$, there exists $\delta >0$ such that 
\begin{equation*}
    \norm{(E(\Bar{t})-E(t_j))d + (B(t_j) - B(\Bar{t}))u} < \epsilon, \quad \forall d\in\D,\: u\in \U,
\end{equation*}
whenever $|t_j - \Bar{t}|<\delta$. Since $t_j \rightarrow \Bar{t}$, there exists $M\in\N$ such that $|t_j-\Bar{t}| < \delta$ for all $j \geq M$. Moreover, since $d\in\D$ is arbitrary, by \eqref{eq: continuity of trimmed control set proof 1} and \eqref{eq: prelim: trimmed input set def},
\begin{align}
    \Bar{\omega} &\in \left\{\omega\in\R^{n}\,|\,\exists e\in\B_{\epsilon}(0) \text{ s.t. } (\omega-e)\in\W(t_j)\right\}\nn\\
    &= \left\{\omega+e\in\R^{n}\,|\, e\in\B_{\epsilon}(0), \omega\in\W(t_j)\right\} = \W(t_j)\oplus \B_{\epsilon}(0), \label{eq: continuity of trimmed control set proof 2}
\end{align}
for all $j \geq M$. Since $\epsilon >0$ is arbitrary, by \eqref{eq: continuity of trimmed control set proof 2} and Assumption \ref{assumption: prelim: trimmed input condition}, there exists a sequence $\{\omega^j\}^\infty_{j=1}$ with $\omega^j\in\W(t_j)$ and $\omega^j\rightarrow \Bar{\omega}$, which implies that $\W$ is open at $\Bar{t}$. 

\emph{Subclaim (b): $\W$ is closed at $\Bar{t}$ (see Definition \ref{def: closed set-valued maps}).} Fix any sequence $\{t_j\}^\infty_{j=1}\subseteq [0,T]$ with $t_j \rightarrow \Bar{t}$ and fix any sequence $\{\omega^j\}^\infty_{j=1}\in\R^{n}$ that converges to $\Bar{\omega}$ with $\omega^j \in \W(t_j)$ for all $j\in\N$. By interchanging $\Bar{t}$ with $t_j$ and $\omega^j$ with $\Bar{\omega}$, the same arguments as subclaim (a) yield
\begin{equation}
    \omega^j \in \W(\Bar{t})\oplus\B_{\epsilon}(0),
    \label{eq: continuity of trimmed control set proof 3}
\end{equation}
for a given $\epsilon >0$ and for all $j$ sufficiently large. As $\W(\Bar{t})$ is a closed set, \eqref{eq: continuity of trimmed control set proof 3} implies $\lim_{j\rightarrow \infty}\omega^j = \Bar{\omega}\in \W(\Bar{t})$, which gives subclaim (b). Consequently, by Definition \ref{def: continuous set-valued map} and subclaims (a) and (b), $\W$ is continuous at $\Bar{t}$. The assertion holds by noting that $\Bar{t} \in [0,T]$ is arbitrary.
\end{proof}

% References
\bibliographystyle{IEEEtran}
\small{\bibliography{references}}
\end{document}